\documentclass[12pt]{amsart}
\usepackage{amsmath}
\usepackage{amsthm}
\usepackage{amsfonts}
\usepackage{hyperref}
\usepackage{epsfig}

\newtheorem{theorem}{Theorem}[section]
\newtheorem{proposition}{Proposition}[section]
\newtheorem{lemma}{Lemma}[section]
\newtheorem{remark}{Remark}[section]
\newtheorem{definition}{Definition}[section]

\renewcommand{\epsilon}{\varepsilon}
\renewcommand{\thesection}{\arabic{section}}

\newcommand{\eqnsection}{
\renewcommand{\theequation}{\thesection.\arabic{equation}}
   \makeatletter
   \csname  @addtoreset\endcsname{equation}{section}
   \makeatother}
\eqnsection

\newcommand{\abs}[1]{\left\vert #1\right\vert}
\newcommand{\R}{\mathbb{R}}
\newcommand{\N}{\mathbb{N}}
\newcommand{\Z}{\mathbb{Z}} 
\newcommand{\PR}{\mathbb{P}}
\newcommand{\ES}{\mathbb{E}}
\newcommand{\LR}{\mathcal{L}}
\newcommand{\IG}{\mathfrak{I}}
\newcommand{\card}{\mathop{\mathrm{card}}}
\newcommand{\1}[1]{{\mathbf 1}{\{#1\}}}

\author[G. Ben Arous]{G\'erard BEN AROUS}
\address{G\'erard Ben Arous, Courant Institute of Mathematical Sciences, New York University, New York, New York 10012
USA} \email{benarous@cims.nyu.edu}

\author[A. Fribergh]{Alexander FRIBERGH}
\address{Alexander Fribergh, Courant Institute of Mathematical Sciences, New York University, New York, New York 10012
USA} \email{fribergh@cims.nyu.edu}

\author[N. Gantert]{Nina GANTERT}
\address{Nina Gantert, Institut f\"ur Mathematische Statistik, Fachbereich Mathematik und Informatik, Universit\"at M\"unster, Einsteinstr. 62, 48149 M\"unster, Germany} \email{gantert@uni-muenster.de}

\author[A. Hammond]{Alan HAMMOND}
\address{Alan Hammond, Department of Statistics, University of Oxford, 1 South Parks road, Oxford, U.K.} \email{hammond@stats.ox.ac.uk}


\begin{document}

\title[]{\quad Biased random walks on \quad\quad\quad\quad \quad\quad\quad\quad Galton-Watson trees with leaves}

\begin{abstract}
We consider a biased random walk $X_n$ on a Galton-Watson tree with leaves in the sub-ballistic regime. We prove that there exists an explicit constant $\gamma= \gamma(\beta) \in (0,1)$, depending on the bias $\beta$, such that $|X_n|$ is of order $n^{\gamma}$. Denoting $\Delta_n$ the hitting time of level $n$, we prove that $\Delta_n/n^{1/\gamma}$ is tight. Moreover we show that $\Delta_n/n^{1/\gamma}$ does not converge in law (at least for large values of $\beta$). We prove that
along the sequences $n_{\lambda}(k)=\lfloor \lambda \beta^{\gamma k}\rfloor$,
$\Delta_n/n^{1/\gamma}$ converges to certain infinitely divisible laws. Key tools for the proof are the classical 
Harris decomposition for Galton-Watson trees, a new variant of regeneration times and the careful analysis of triangular arrays of i.i.d.~heavy-tailed random variables.
\end{abstract}

\maketitle

\section{Introduction and statement of the results}
\label{intro}

Consider a supercritical Galton-Watson branching process with generating function ${\bf f}(z)=\sum_{k \geq 0} p_k z^k$, i.e.~the offspring of all individuals are i.i.d.~copies of $Z$, where ${\bf P}[Z=k]=p_k$. We assume that the tree is supercritical and has leaves, i.e. ${\bf m}:={\bf E}[Z]={\bf f}'(1)\in (1,\infty)$ and $p_0>0$.
We denote by $q \in (0,1)$ the extinction probability, which is characterized by ${\bf f}(q)=q$. Starting from a single progenitor called root and denoted by $0$, this process yields a random tree $T$. We will always condition on the event of non-extinction, so that $T$ is an infinite random tree. We denote $(\Omega,{\bf P})$ the associated probability space:  
${\bf P}$ is the law of the original tree, conditioned on non-extinction. For a vertex $u \in T$, we denote by $\abs{u}=d(0,u)$ the distance of $u$ to the root.


For $\omega \in \Omega$, on the infinite Galton-Watson tree $T(\omega)$, we consider the
$\beta$--biased random walk as in \cite{LPP}. More precisely, we define, for $\beta>1$, a Markov chain $(X_n)_{n\in \N}$ on the vertices of $T$, such that if $u\neq 0$ and $u$ has $k$ children $v_1,\ldots, v_k$ and parent $\overleftarrow{u}$, then
\begin{enumerate}
\item $P[X_{n+1}=\overleftarrow{u}|X_n=u]=\frac 1 {1+ \beta k}$, 
\item $P[X_{n+1}= v_i |X_n=u] = \frac {\beta}{1+\beta k}$, for $1\leq i\leq k$,
\end{enumerate}
and from 0 all transitions to its children are equally likely. This is a reversible Markov chain, and as such, can be described as an electrical network with conductances $c(\overleftarrow{x},x):=\beta^{\abs{x}-1}$ on every edge of the tree (see ~\cite{LP} for background on electrical networks). 

We always take $X_0 = 0$, that is we start the walk from the root of the tree. We denote by $P^{\omega}[\cdot]$ 
the law of $(X_n)_{n=0,1,2, \ldots}$ and we define the averaged law as the semidirect product $\PR={\bf P} \times P^{\omega}$.

Many interesting facts are known about this walk (see~\cite{LPP}). As one might expect, it is transient. 
It is known that $\PR$-a.s., $|X_n|/n$ converges to a deterministic limit $v$. Moreover, the random walk is \emph{ballistic}, i.e. its limiting velocity $v >0$, if and only if $\beta < \beta_c = 1/{\bf f}'(q)$. In the \emph{subballistic} regime, i.e. if $\beta \geq \beta_c$, we have $v=0$. The reason for the subballistic regime is that the walk loses time in traps of the tree, from where it cannot go to infinity without having to go for a long time against the drift which keeps it into the trap. The hypothesis $p_0>0$ is crucial for this to happen.

As in all subballistic models, a natural question comes up: what is the typical distance of the walker from the root after $n$ steps?
This is the question we address in this paper.
We always assume that
\[
{\bf E}[Z^2] <\infty
\]
and 
\[
\beta > 1/ {\bf f}'(q)\, ,
\]
recalling that $1/ {\bf f}'(q) > 1$. We introduce the exponent
\begin{equation}{\label{notationgamma}}\gamma:= \frac{-\ln {\bf f}'(q)}{\ln \beta}=\frac{\ln \beta_c}{\ln \beta} < 1
\end{equation}
so that $\beta^\gamma =  1/ {\bf f}'(q)$.

Let $\Delta_n$ be the hitting time of the $n$-th level:
\[
\Delta_n=\inf\{i\geq 0: \abs{X_i}=n\}.
\]

\begin{theorem}
\label{tightness}
(i) The laws of $(\Delta_n/n^{1/\gamma})_{n\geq 0}$ under $\PR$ are tight.\\
(ii) The laws of $(|X_n|/n^{\gamma})_{n\geq 0}$ under $\PR$ are tight.\\
(iii) We have
\begin{equation}\label{magni}
\lim\limits_{n \to \infty} \frac{\ln |X_n|}{\ln n}=\gamma,\ \PR-\text{a.s.}
\end{equation}
\end{theorem}

Of course, this raises the question of convergence in distribution of the sequence $(\Delta_n/n^{1/\gamma})_{n\geq 0}$.
The next theorem gives a negative answer.

\begin{theorem}
\label{non-conve}
For $\beta$ large enough, the sequence $(\Delta_n/n^{1/\gamma})_{n\geq 0}$ does not converge in distribution.
\end{theorem}

However, we can establish convergence in distribution along certain subsequences.

\begin{theorem}
\label{subsequ}
For any $\lambda>0$, denoting $n_{\lambda}(k)=\lfloor \lambda {\bf f}'(q)^{-k} \rfloor$, we have
\[
\frac{ \Delta_{n_{\lambda}(k)}} {n_{\lambda}(k)^{1/\gamma}} \xrightarrow{d} Y_\lambda
\]
where the random variable $Y_\lambda$ has an infinitely divisible law $\mu_\lambda$.
\end{theorem}

We now describe the limit laws $\mu_\lambda$. For some constants $\rho$ and $C_a$ (the constant $\rho$ is defined in~(\ref{notationrho}), the constant $C_a$ in Lemma~\ref{trapproba}),
we have 
\[
Y_\lambda = (\rho C_a\lambda)^{1/\gamma}\,\,  \widetilde{Y}_{(\rho C_a\lambda)^{1/\gamma}}
\]
where 
\[
\widetilde{Y}_\lambda \text { has the law } \IG(d_\lambda,0,\LR_\lambda)\, .
\]
The infinitely divisible law 
$\IG(d_{\lambda},0,\LR_\lambda)$
is given by its L\'evy representation (see \cite{Petrov}, p. 32).
More precisely, the characteristic function of $\IG(d_{\lambda},0,\LR_\lambda)$ can be written in the form
\[
\ES\left [e^{it\widetilde{Y}_\lambda}\right] = \int e^{itx}\IG(d_{\lambda},0,\LR_\lambda) (dx) = 
\exp\left(id_\lambda t + \int\limits_0^\infty \left(e^{itx}-1 - \frac{itx}{1+x^2}\right)d\LR_\lambda(x)\right)
\]
where $d_\lambda$ is a real constant and $\LR_\lambda$ a real function which is non-decreasing on the interval $(0, \infty)$ and satisfies $\LR_\lambda(x) \to 0$ for $x \to \infty$ and $\int \limits_0^a x^2 d\LR_\lambda(x) < \infty$ for every $a > 0$. Comparing to the general representation formula in \cite{Petrov}, p. 32, we here have that the gaussian part vanishes and $\LR_\lambda(x) =0$ for $x < 0$. The function $\LR_\lambda$ is called the L\'evy spectral function.
Note that $\LR_\lambda$ is not a L\'evy-Khintchine spectral function.

In order to describe $\LR_\lambda$, define the random variable
\begin{equation}\label{Zdef}
\mathcal{Z}_{\infty}=\frac{S_{\infty}}{1-\beta^{-1}} \sum_{i=1}^{\text{Bin}(W_{\infty},p_{\infty})} {\bf e}_i,
\end{equation}
where $p_{\infty}=1-\beta^{-1}$ is the escape probability of a $\beta$-biased random walk on $\N$. Further, the random variables ${\bf e}_i$ in (\ref{Zdef}) are i.i.d.~exponential random variables of parameter 1 and the non-negative random variables $({\bf e}_i)$, $W_{\infty}$ and $S_\infty$ in (\ref{Zdef}) are independent.
The random variables $S_{\infty}$ and $W_{\infty}$ will be described in (\ref{notationSinfty}) and Proposition \ref{Winfty} respectively. The random variable $\text{Bin}(W_{\infty},p_{\infty})$ has a Binomial law with parameters $W_\infty$ and $p_\infty$. Now, denoting by $\overline{F}_{\infty}(x)=\PR[\mathcal{Z}_{\infty}>x]$ the tail function of $\mathcal{Z}_{\infty}$, we have
\begin{theorem}\label{infdivprop}
For all $\lambda > 0$, the following statements hold.

(i) The L\'evy spectral function $\LR_1$ is given as follows.
\[
\LR_1(x)=\left\{ \begin{array}{ll}
                         0 & \text{ if $x<0$,}\\
                         \displaystyle{-(1-\beta^{-\gamma})\sum_{k\in \Z}\beta^{\gamma k}\overline{F}_{\infty}(x \beta^k)} & \text{ if $x>0$.}
                         \end{array} \right.
\]
In particular, $\LR_1(x) = \beta^\gamma \LR_1(\beta x)$.

(ii) For all $x\in \R$, \quad $\LR_{\lambda}(x)=\lambda^{\gamma} \LR_1(\lambda x)$. In particular,
$\LR_{\beta^j}(x)=\LR_1(x)$, for all integers $j$.\\
(iii) $d_\lambda$ is given by
\[
d_{\lambda}=\lambda^{1+\gamma}(1-\beta^{-\gamma})\sum_{k\in \Z}\beta^{(1+\gamma)k}E\Bigl[\frac {\mathcal{Z}_{\infty}}{(\lambda\beta^k)^2+\mathcal{Z}_{\infty}^2}\Bigr].
\]
(iv) $\LR_\lambda$ is absolutely continuous.\\
(v) The following bounds hold
\begin{equation}\label{Lbounds}
\frac{1}{\beta^\gamma}\ES[\mathcal{Z}_\infty^\gamma]\frac{1}{x^\gamma} \leq - \LR_1(x) \leq \ES[\mathcal{Z}_\infty^\gamma]\frac{1}{x^\gamma}\, .
\end{equation}
(vi) The measure $\mu_\lambda$ is absolutely continuous with respect to Lebesgue measure and has a moment of order $\alpha$ if and only if $\alpha < \gamma$. \\
(vii) When $\beta$ is large enough, $x^\gamma \LR_\lambda(x)$ is not a constant.\\
(viii) The random variable $\mathcal{Z}_{\infty}$ has an atom at $0$ and a smooth density $\psi$ on $(0, \infty)$. Further, $\mathcal{Z}_{\infty}$ has finite expectation. 
\end{theorem}

\begin{remark}
We believe that Theorem \ref{non-conve} holds true for all values $\beta > \beta_c$. The proof would amount to showing that the 
function $x^\gamma \LR_1(x)$, with $\LR_1(x)$ given in Theorem \ref{infdivprop}, is not a constant.
\end{remark}

Next we explain briefly, using a toy example, the reason for the non-convergence of $(\Delta_n/n^{1/\gamma})_{n\geq 0}$ and the convergence of subsequences in Theorems \ref{non-conve} and \ref{subsequ}.
The reasons lie in the classical theory of sums of i.i.d.~random variables. Consider a sequence of i.i.d.~random variables $G_i$,
geometrically distributed with parameter $a$.
Let
\[
S_n = \sum\limits_{i=1}^n \beta^{G_i}\, .
\]
It is easy to see, using classical results about triangular arrays of i.i.d.~random variables (c.f. \cite{Petrov}), that for $\alpha = \frac{- \log(1-a)}{\log \beta}$, and $n_\lambda(k) = \beta^{-\alpha k}$, the distributions of
\[
\frac{1}{n_\lambda(k)^{1/\alpha}}S_{n_\lambda(k)} \text{ converge to an infinitely divisible law}
\]
(see Theorem \ref{sum_iid} for a more general result). But obviously here $S_n/ n^{1/ \alpha}$ cannot converge in law, if $\alpha < 2$, because one easily checks that the distribution of $\beta^{G_1}$ does not belong to the domain of attraction of any stable law.
This is the basis of our belief that Theorem \ref{non-conve} should be valid for any $\beta > \beta_c$.

We now discuss the motivation for this work. If one considers a biased random walk on a supercritical percolation cluster on $\Z^d$, it is known that, at low bias, the random walk is ballistic (i.e. has a positive velocity) and has gaussian fluctuations, see ~\cite{Sznitman} and ~\cite{Berger}. It is also known that, at strong bias, the random walk is subballistic (i.e. the velocity vanishes). It should be noted that, in contrast to the Galton-Watson tree, the existence of a critical value separating the two regimes is not established for supercritical percolation clusters. The behaviour of the (law of) the random walk in the subballistic regime is a very interesting open problem. It was noted in \cite{AlainHouches} that the behaviour of the random walk in this regime is reminiscent of trap models introduced by Bouchaud (see \cite{Bouchaud} and \cite{BenarousCerny}). Our work indeed substantiates this analogy in the simpler case of supercritical random trees. We show that most of the time spent by the random walk before reaching level $n$ is spent in deep traps. These trapping times are roughly independent and are heavy-tailed. However, their distribution does not belong to the domain of attraction of a stable law, which explains the non-convergence result in Theorem \ref{non-conve}.

We note that it is possible to obtain convergence results to stable laws if one gets rid of the inherent lattice structure. One way to do this is to randomize the bias $\beta$. This is the approach of the forthcoming paper \cite{BenarousHammond}.

For other recent interesting works about random walks on trees, we refer to ~\cite{Hu}, ~\cite{aidekon}, and~\cite{PZ}. 

There is also an analogy with the one-dimensional random walk in an i.i.d.~random environment (RWRE). 
This model also shows a ballistic and a subballistic regime, explicitly known in terms of the parameters of the model. We refer to \cite{Zeitouni} for a survey. In the subballistic regime, it was shown in~\cite{KKS} that depending on a certain parameter $\kappa \in (0,1]$, and under a non-lattice assumption, $\frac{X_n}{n^{\kappa}}$ converges to a functional of a stable law, if $\kappa<1$, and $\frac{X_n}{n/ \ln n}$ converges to a functional of a stable law, if $\kappa=1$.
Recently, using a precise description of the environment, \cite{ESZ} and~\cite{ESZ2} refined this last theorem by describing all the parameters of the stable law, in the case $\kappa < 1$.

Our method has some similarity to the one used in~\cite{ESZ2}. In comparison to~\cite{ESZ2}, an additional difficulty arises from the fact the traps met depend not only on the environment but also on the walk. Moreover one has to take into account the number of times the walker enters a trap, which is a complicated matter because of the inhomogeneity of the tree. This major technical difficulty can be overcome by decomposing the tree and the walk into independent parts, which we do using a new variant of regeneration times.

The paper is organized as follows: In Section~\ref{sectionconstruction} and Section~\ref{trapcons} we explain how to decompose the tree and the walk. In Section~\ref{sketch} we give a sketch of the proof of Theorem \ref{subsequ}. Sections \ref{onlybig} - \ref{sectiontail} prepare the proof of Theorem \ref{subsequ} and explain why the hitting time of level $n$ is comparable to a sum of i.i.d.~random variables. Section \ref{sect_sum_iid} is self-contained and its main result, Theorem \ref{sum_iid}, is a classical limit theorem for sums of i.i.d.~random variables which is tailored for our situation. In Section \ref{lim-thm}, we finally give the proofs of the results. In Subsection \ref{subseq}, we apply Theorem \ref{sum_iid} to prove 
Theorem \ref{subsequ}. Subsection \ref{non-conv} is devoted to the proof of Theorem \ref{non-conve}, Subsection \ref{magnit} gives the proof of Theorem \ref{tightness} and Subsection \ref{limprop} the proof of Theorem \ref{infdivprop}. 

Let us give some conventions about notations. The parameters $\beta$ and $(p_k)_{k\geq 0}$ will remain fixed so we will usually not point out that constants depend on them. Most constants will be denoted $c$ or $C$ and their value may change from line to line to ease notations. Specific constants will have a subscript as for example $C_a$. We will always denote by $G(a)$ a geometric random variable of parameter $a$, with law given by $P[G_a\geq k]=(1-a)^{k-1}$ for $k \geq 1$.

\section{Constructing the environment and the walk in the appropriate way}{\label{sectionconstruction}}

In order to understand properly the way the walk is slowed down, we need to decompose the tree. Set 
\begin{equation}
\label{notationgh}
{\bf g}(s)=\displaystyle{\frac{{\bf f}((1-q)s+q )-q }{1-q}} \text{ and } {\bf h}(s)=\frac{{\bf f}(qs)}q.
\end{equation}

It is known (see~\cite{lycap}), that a ${\bf f}$-Galton-Watson tree (with $p_0>0$) can be generated by
\begin{itemize}
\item[(i)] growing a ${\bf g}$-Galton-Watson tree $T_{{\bf g}}$ called the backbone, where all vertices have an infinite line of descent,
\item[(ii)] attaching on each vertex $x$ of $T_{{\bf g}}$ a random number $N_x$ of ${\bf h}$-Galton-Watson trees, acting as traps in the environment $T$,
\end{itemize}
where $N_x$ has a distribution depending only on $\text{deg}_{T_{{\bf g}}}(x)$ and given $T_{{\bf g}}$ and $N_x$ the traps are i.i.d., see \cite{lycap} for details.

\begin{figure}[htpd]
\centering 
\epsfig{file=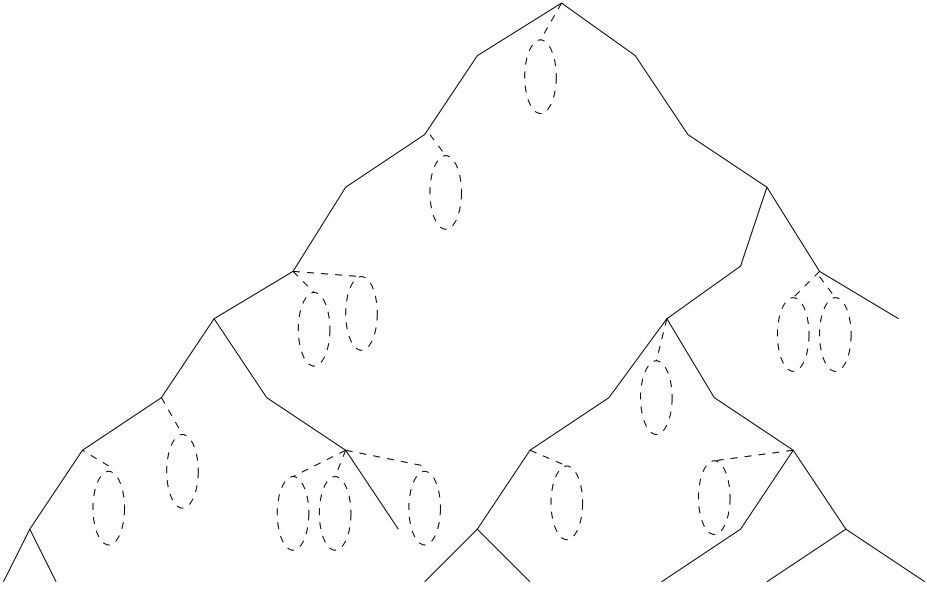, scale=0.7}
\caption{The Galton-Watson tree is decomposed into the backbone (solid lines) and the traps (dashed lines).}
\end{figure}

We will refer to any vertex at distance exactly one of the backbone as a bud. It is important to consider the backbone together with the buds to understand the number of visits to traps.

It will be convenient to
consider the attached Galton-Watson trees together with the edge which connects them to the backbone.
We define a \emph{trap} to be a graph $(x\cup V, [x,y] \cup E)$, where $x$ is a vertex of the backbone, $y$ is a bud adjacent to $x$ and $V$ (resp. $E$) are the vertices (resp. edges) of the descendants of $y$. The traps can themselves be decomposed in  a portion of $\Z$ called the spine, to which smaller trees called subtraps are added: this construction is presented in detail in Section~\ref{trapcons}.

Let us now construct the random walk. We need to consider the walk on the backbone and on the buds, to this end we introduce 
\begin{enumerate}
\item $\sigma_0=\sigma_0'=0$,
\item $\sigma_{n+1}= \inf \{i > \sigma_n | X_{i-1}, X_{i} \in \text{backbone} \}$,
\item $\sigma_{n+1}'=\inf \{i > \sigma_n'|X_{i-1}, X_i \in \text{backbone} \cup \text{buds} \}$,
\end{enumerate}
and we define $Y_n=X_{\sigma_n}$ the embedded walk on the backbone, respectively $Y_n'=X_{\sigma_n'}$ the embedded walk on the 
backbone and the buds.

Moreover define $\Delta_n^{Y}=\card \{i \geq 0: \sigma_i \leq \Delta_n\}$ the time spent on the backbone to reach level $n$ and similarly $\Delta_n^{Y'}=\card \{i \geq 0: \sigma_i' \leq \Delta_n\}$. 

Denote, for a set $A$ in the tree, $T_A^+=\min \{n\geq 1 | X_n \in A \}$,  $T_y^+:=T_{\{y\}}^+$, $T_A:=\min \{n\geq 0 | X_n \in A \}$, and $T_y:=T_{\{y\}}$. 

Note that the process $(Y_n)_{n\geq 0}$ is a Markov chain on the backbone, which is independent of the traps and the time spent in the traps. Here one has to be aware that visits to $root$ do not count as ``time spent in a trap'', precise definitions will follow below. Hence, in order to generate $Y_n$ we use a sequence of i.i.d.~random variables $U_i$ uniformly distributed on $[0,1]$. If $Y_j=w$ with $Z_1^*$ children on the backbone, then
\begin{enumerate}
\item $Y_{i+1}=\overleftarrow{w} $, if $U_i \in \Bigl[0,\frac 1 {Z_1^*\beta+1}\Bigr]$,
\item $Y_{i+1}=$ the $j^{th}$-child of $w$, if $U_i\in \Bigl[1- \frac {j\beta} {Z_1^* \beta+1}, 1- \frac {(j-1)\beta}{Z_1^* \beta +1}\Bigr].$
\end{enumerate}

For background on regeneration times we refer to~\cite{SZ} or~\cite{Zeitouni}. In the case of a $\beta$-biased random walk $\widetilde{Y}_n$ on $\Z$, a time $t$ is a regeneration time if
\[
\widetilde{Y}_t > \max_{s<t} \widetilde{Y}_s \text{ and } \widetilde{Y}_t< \min_{s>t} \widetilde{Y}_s.
\]

\begin{definition} A time $t$ is a super-regeneration time for $Y_n$, if $t$ is a regeneration time for the corresponding $\beta$-biased random walk $\widetilde{Y}_n$ on $\Z$ defined by 
\begin{itemize}
\item[(i)] $\widetilde{Y}_0=0,$
\item[(ii)] $\widetilde{Y}_{n+1}=\widetilde{Y}_n-1$, if $U_n \in \Bigl[0,\frac{1}{\beta+1}\Bigr]$,
\item[(iii)] $\widetilde{Y}_{n+1}=\widetilde{Y}_n+1$ otherwise.
\end{itemize}
\end{definition}

We denote $t-\text{SR}$ the event that $t$ is a super-regeneration time for $Y_n$.

It is obvious that a super-regeneration time for $Y_n$ is a regeneration time for $Y_n$ in the usual sense (the converse is false). 

The walk can then be decomposed between the successive super-regeneration times 
\begin{itemize}
\item[(i)] $\tau_0=0$,
\item[(ii)] $\tau_{i+1}= \inf\{j \geq \tau_i : j-\text{SR}\}$.
\end{itemize}

Since the regeneration times of a $\beta$-biased random walk on $\Z$ have some exponential moments, there exists $a>1$ such that $\ES[a^{\tau_2-\tau_1}]<\infty$ and $\ES[a^{\tau_1}]<\infty$.

\begin{remark}{\label{avantregen}}
The advantage of super-regeneration times compared to classical regeneration times is that the presence of a super-regeneration time does not 
depend on the environment, but only the on the sequence $(U_i)_{i\geq 0}$.
\end{remark}

\begin{remark}{\label{probregen}}
The drawback of super-regeneration times is that the event that $k$ is a super-regeneration time depends on the random variables $(U_i)_{i\geq 0}$ and not only on the trajectory of the random walk $(Y_n)_{n \geq 0}$.
\end{remark}

Denoting for $k\geq 1$, the $\sigma$-field
\[
\mathcal{G}_k = \sigma(\tau_1,\ldots,\tau_k,(Y_{n\wedge \tau_k})_{n\geq 0},\ \{x\in T(\omega), x \text{ is not a descendant of } Y_{\tau_k}\}).
\]

We have the following proposition
\begin{proposition}{\label{regenstruc1}}
For $k\geq 1$,
\begin{align*}
& \PR[(Y_{\tau_k+n}-Y_{\tau_k})_{n\geq 0} \in \cdot,\ \{x \in T(\omega), x \text{ is a descendant of } Y_{\tau_k}\}\in \cdot|\mathcal{G}_k]\\ = & \PR[(Y_n)_{n \geq 0} \in \cdot, T(\omega) \in \cdot |0-\text{SR}]. 
\end{align*}
\end{proposition}

\begin{remark}{\label{regenstruc2}} The conditioning $0-\text{SR}$ refers only to the walk on the backbone, hence it is obvious that the behaviour of the walk in the traps and the number of times the walker enters a trap is independent of that event. \end{remark}

We skip the proof of this Proposition since it is standard. A consequence of the proposition is that the environment and the walk can be subdivided into super-regeneration blocks which are i.i.d.~(except for the first one). As a consequence we have that
\begin{equation}
\label{notationrho}
\rho_n:=\frac{ \card\{ Y_1, \ldots , Y_{\Delta_n^{Y}} \}} n \text{ satisfies } \rho_n\to \rho:=\frac{\ES[\card\{Y_{\tau_1},\ldots, Y_{\tau_2 -1}\}]}{\ES[\tau_2 - \tau_1] }, \quad \PR-a.s
\end{equation}
which is the average number of vertices per level visited by $Y_n$. This quantity is finite since it is bounded above by than $1/v(\beta)$, where $v(\beta)$ is the speed of $|Y_n|$ which is strictly positive by a comparison to the $\beta$-biased random walk on $\Z$. 

When applying the previous proposition, it will be convenient to use the time-shift for the random walk, which we will denote by $\theta$.

\section{Constructing a trap}{\label{trapcons}}

In the decomposition theorem for Galton-Watson trees, we attach to the vertices of the backbone a (random) number of ${\bf h}$-Galton-Watson trees. We will denote their distribution with ${\bf Q}$, hence ${\bf Q}[Z=k]=q_k:=p_kq^{k-1}$, where $Z$ denotes the number of children of a given vertex. As stated before the object we will denote a trap has an additional edge: to describe a trap $\ell$ we take a vertex called $root$ (or $root(\ell)$ to emphasize the trap), link it to another vertex (denoted $\overrightarrow {root}(\ell)$), which is the actual root of a random ${\bf h}$-Galton-Watson tree. 

When we use random variables associated to a trap, we refer to the random part of that trap (the ${\bf h}$-Galton-Watson tree). For example the notation $Z_n$ is the number of children at the generation $n$ with $\overrightarrow{root}$ being generation $0$.  In particular, we introduce the height of a trap
\begin{equation}
\label{notationH}
H =\max \{n \geq 0, Z_n >0 \},
\end{equation}
and we say a trap has height $k$ if $H(\ell)=k$, i.e.~the distance between $\overrightarrow{root}$ and the bottom point of the trap is $k$.

This way of denoting the random variables has the advantage that $Z_n$ (resp. $H$) are distributed under ${\bf Q}$, as the number of children at generation $n$ (resp. the height) of an ${\bf h}$-Galton-Watson tree.

The biggest traps seen up to level $n$ are of size $-\ln n / \ln {\bf f}'(q)$, therefore a trap will be considered big if its height is greater or equal to 
\begin{equation}
\label{notationhn}
h_n= \Bigl\lceil(1-\epsilon) \frac{\ln n}{-\ln {\bf f}'(q)}\Bigr\rceil,
\end{equation}
for some $\epsilon>0$ which will eventually be chosen small enough. Such a trap will be called an $h_n$-trap or a big trap. It is in those traps that the walker will spend the majority of his time and therefore it is important to have a good description of them.

The traps are (apart from the additional edge) subcritical Galton-Watson trees, as such, they can be grown from the bottom following a procedure described in~\cite{Geiger}, that we recall for completeness. We will denote by $\delta$ the starting point of the procedure, corresponding to the leftmost bottom point of the trap, this last notation will be kept for the whole paper.

With a slight abuse of notation, we will denote by ${\bf Q}$ a probability measure on an enlarged probability space containing the following additional information.

We denote by $(\phi_{n+1},\psi_{n+1})$ with $n\geq 0$, a sequence of i.i.d pairs of random variables with joint law given by
\begin{equation}
\label{notationPSI}
{\bf Q}[\phi_{n+1}=j,\psi_{n+1}=k]=c_n q_k {\bf Q}[Z_n=0]^{j-1} {\bf Q}[Z_{n+1}=0]^{k-j},\ 1\leq j \leq k, \ k \geq 1,
\end{equation}
where $c_n =\frac{{\bf Q}[H=n]}{{\bf Q}[H=n+1]}$.

Set $\mathcal{T}_0=\{\delta\}$. Construct $\mathcal{T}_{n+1}$, $n\geq 0$ inductively as follows: 
\begin{enumerate}
\item let the first generation size of $\mathcal{T}_{n+1}$ be $\psi_{n+1}$,
\item let $\mathcal{T}_n$ be the subtree founded by the $\phi_{n+1}$-th first generation vertex of $\mathcal{T}_{n+1}$,
\item attach independent ${\bf h}$-Galton-Watson trees which are conditioned on having height stricly less than $n$ to the $\phi_{n+1}-1$ siblings to the left of the distinguished first generation vertex,
\item[(iv)] attach independent ${\bf h}$-Galton-Watson trees which are conditioned on having height strictly less than $n+1$ to the $\psi_{n+1}-\phi_{n+1}$ siblings to the right of the distinguished first generation vertex.
\end{enumerate}

Then $\mathcal{T}_{n+1}$ has the law of an ${\bf h}$-Galton-Watson tree conditioned to have height $n+1$ (see~\cite{Geiger}).

We denote $\mathcal{T}$ the infinite tree asymptotically obtained by this procedure; from this tree we can obviously recover all $\mathcal{T}_n$. If we pick independently the height $H$ of an ${\bf h}$-Galton-Watson tree and the infinite tree $\mathcal{T}$ obtained by the previous algorithm, then $\mathcal{T}_H$ has the same law as an ${\bf h}$-Galton-Watson tree.

We will call spine of this Galton-Watson tree the ancestors of $\delta$. If $y\neq \delta$ is in the spine, $\overrightarrow{y}$ denotes its only child in the spine. We define a subtrap to be a graph $(x\cup V, [x,y] \cup E)$, where $x$ is a vertex of the spine, $y$ is a descendant of $x$ not on the spine and $V$ (resp. $E$) are the vertices (resp. edges) of the descendants of $y$. The vertex $x$ is called the root of the subtrap, and we denote 
\begin{equation}
\label{Sdef}
S_x \text{  the set of all subtraps rooted at } x\, .
\end{equation}
\begin{figure}[htpd]
\centering 
\epsfig{file=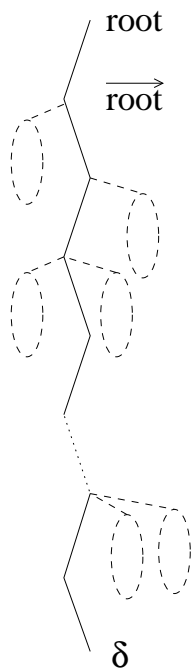, scale=0.7}
\caption{The trap is decomposed into the spine (solide lines) and the subtraps (dashed lines).}
\end{figure}

We denote by $S_{n+1}^{i,j,k}$ and $\Pi_{n+1}^{i,j,k}$ with $n,i,j \geq 0$ and $k=1,2$, two sequences of independent random variables, which are independent of $(\phi_n,\psi_n)_{n\geq 0}$ and given by
\begin{enumerate}
\item $S_{n}^{n+1,j,1}$ (resp. $S_{n}^{n,j,2}$) is the $j$-th subtrap conditioned to have height less than $n$ added on the left (resp. right) of the $n+1$-th (resp. $n$-th) ancestor of $\delta$,
\item $\Pi_{n}^{i,j,k}$ is the weight of $S_{n}^{i,j,k}$ under the invariant measure associated to the conductances $\beta^{i+1}$ between the level $i$ and $i+1$, the root of $S_{n}^{i,j,k}$ being counted as level 0.
\end{enumerate}
These random variables describe the subtraps and their weights.

We denote $\Pi_{-1}^{i,j,k}=0$ and
\begin{equation}
\label{notationLi}\Lambda_i(\omega)=\sum_{j=1}^{\phi_{i}-1} \Pi_{i-1}^{i,j,1}+\sum_{j=1}^{\psi_{i}-\phi_{i}} \Pi_i^{i,j,2},
\end{equation}
which is the weight of the subtraps added to the $i-$th ancestor of $\delta$.

Due to the next lemma, the random variables $\Lambda_i$ will be important to describe the time spent in traps.
We recall that a reversible Markov chain can be described as an electrical network. An electrical network is a connected graph $G= (V, E)$ with positive conductances $c(e)_{e \in E}$ on its edges. It defines a Markov chain with transition probabilities
$p(x,y) = c([x,y])/ \sum_{z: [x,z]\in E}c([x,z])$, where we denote $[x,y]$ the edge connecting $x$ and $y$ (and $p(x,y)=0$ if there is no such edge). This Markov chain has $\widehat{\pi}(x)=\sum_{[x,y] \in E} c([x,y])$, $x\in V$, as an invariant measure. If it is positive recurrent, the unique invariant probability measure is given by $\widehat{\pi}(x)/\widehat{\pi}(V)$, where 
$\widehat{\pi}(V) = \sum_{y\in V}\widehat{\pi}(y)$. We refer to \cite{LP} for background on electrical networks, and we recall the following lemma.
\begin{lemma}
\label{meanreturntime}
Let $(G,c(e)_{e \in E})$ be a positive recurrent electrical network, $x \in V$ and $P_x$ the law of the Markov chain started at $x$. 
If $\sum_{z:[x,z] \in E} c([x,z])=1$, then
\[
E_x [T_x^+]= 2 \sum_{e \in E} c(e).
\]
\end{lemma}

\begin{proof}
The expected number of visits to $y$ before returning to $x$ is 
$\widehat \pi(y)/ \widehat \pi(x)$. Hence,
\[
E_{x}[T_{x}^+]=\frac{\widehat{\pi}(V)}{\widehat{\pi}(x)}=\widehat{\pi}(V),
\]
since $\widehat{\pi}(x)=\sum_{z: [x,z] \in E} c([x,z])=1$. 
Then we simply notice that $\widehat{\pi}(V)=2 \sum_{e \in E} c(e)$.
\end{proof}

Let us introduce another important random variable
\begin{equation}
\label{notationSinfty}S_{\infty}=2 \sum_{i=0}^{\infty} \beta^{-i} (1+\Lambda_i),
\end{equation}
which appears in the statement of Theorem \ref{infdivprop}. It is the mean return time to $\delta$ of the walk on the infinite tree $\mathcal{T}$ described in the algorithm following~(\ref{notationPSI}).

\begin{lemma}
\label{W}
There exists a constant $C_{\psi}$ depending on $(p_k)_{k\geq 0}$, such that for $n\geq 0$ and $k\geq 0$, 
\[
{\bf Q}[\psi_{n+1}=k]\leq C_{\psi} k q^k.
\]

In particular, for another constant ${\widetilde C}_\psi$,  
$\displaystyle{\sup_{i\in \N}} E_{\bf Q}[\psi_i] \leq {\widetilde C}_\psi < \infty$.
\end{lemma}

\begin{proof}
Recalling~(\ref{notationPSI}), we get
\begin{align*}
{\bf Q}[\psi_{n+1}=k] &=\sum_{j=1}^k {\bf Q}[\phi_{n+1}=j,\psi_{n+1}=k] \\
              &=c_n q_k   \sum_{j=1}^k {\bf Q}[Z_n=0]^{j-1} {\bf Q }[Z_{n+1}=0]^{k-j} \\
              &\leq c_n k q_k.
\end{align*}

It is enough to show that the sequence $(c_n)_{n \geq 0}$ is bounded from above. A Galton-Watson tree of height $n+1$ can be obtained as $root$ having $j$ children, one of which produces a Galton-Watson tree of height $n$, the others having no children of their own. Thus
\[
1/c_n  = {\bf Q}[H=n+1]/{\bf Q}[H= n] \geq q_j q_0^{j-1},
\]
for any $j\geq 1$. We fix $j_0\geq 1$ so that $q_{j_0}>0$ and we get 
\[
{\bf Q}[\psi_{n+1}=k] \leq \frac 1 {q_{j_0} q_0^{j_0-1}} k q^{k-1},
\]
where we used $q_k =p_k q^{k-1} \leq q^{k-1}$.
\end{proof}

Using this lemma we can get a tail estimate for the height of traps. 

\begin{lemma}
\label{tailH}
There exists $\alpha>0$ such that
\[
{\bf Q}[H\geq n] \sim \alpha {\bf f}'(q)^n.
\]
\end{lemma}

\begin{proof}
It is classical (see~\cite{Heathcote}) that for any Galton-Watson tree of law $\widetilde Q$ with $E_{\widetilde Q}[Z_1]=m<1$ expected number of children, we have 
\[
\lim_{n \to \infty} \frac{\widetilde Q[Z_n>0]}{m^n}>0 \Longleftrightarrow E_{{\widetilde Q}}[Z_1 \log^+Z_1]<\infty.
\]

The integrability condition is satisfied for ${\bf Q}$ since $q_k=p_kq^{k-1}\leq q^{k-1}$, and the result follows.
\end{proof}

We also recall the following classical upper bound 
\begin{equation}
\label{tailmaj}{\bf Q}[H\geq n]={\bf Q}[Z_n>0] = {\bf Q}[Z_n \geq 1 ]\leq  E_{\bf Q}[Z_n]={\bf f}'(q)^n.
\end{equation}

The following lemma seems obvious, but not standard, so we include its proof for the convenience of the reader.

\begin{lemma}{\label{Z}}
We have for $k \geq 0$,
\[
{\bf Q}[Z_1 \leq  k|Z_n=0] \geq  {\bf Q}[Z_1 \leq k].
\]

In particular $E_{\bf Q}[Z_i |Z_n=0] \leq {\bf f}'(q)^i,$ for any $i\geq 0$ and $n\geq 0$.
\end{lemma}

\begin{proof}
Denoting $D_n$ a geometric random variable of parameter $1-{\bf Q}[Z_{n-1}=0]$ which is independent of $Z_1$, we have ${\bf Q}[Z_1\leq k |Z_{n}=0]= 
{\bf Q}[Z_1\leq k | Z_1 < D_n]$. Then compute
\begin{align*}
{\bf Q}[Z_1\leq k | Z_1 < D_n] & = \frac{\sum_{j=0}^k {\bf Q}[D_n > j]  {\bf Q}[Z_1=j]}{\sum_{j=0}^{\infty} {\bf Q}[D_n > j] {\bf Q}[Z_1=j]} \\
& = \left(1+ \frac {\sum_{j=k+1}^{\infty} {\bf Q}[D_n> j] {\bf Q}[Z_1=j]}{\sum_{j=0}^k {\bf Q}[D_n> j] {\bf Q}[Z_1=j]}\right)^{-1},
\end{align*}
now use that for all $j'<k<j$ we have ${\bf Q}[D_n> j] \leq {\bf Q}[D_n > k] \leq {\bf Q}[D_n> j']$, yielding 
\[
{\bf Q}[Z_1\leq k | Z_1 < D_n] \geq \frac{\sum_{j=0}^k  {\bf Q}[Z_1=j]}{\sum_{j=0}^{\infty}  {\bf Q}[Z_1=j]}= {\bf Q}[Z_1 \leq k].
\]
\end{proof}

We can now estimate $E_{\bf Q}[\Lambda_i]$.
\begin{lemma}
\label{L}
For all $i\geq 0$,
\[
E_{\bf Q}[\Lambda_i] \leq \frac{{\widetilde C}_{\psi}}{1-(\beta {\bf f}'(q))^{-1}}({\bf f}'(q) \beta)^i.
\]
\end{lemma}

\begin{proof}
Using~(\ref{notationLi}), Lemma~\ref{W} and Lemma~\ref{Z}, we get
\[
E_{\bf Q}[\Lambda_i]=E_{\bf Q}[\phi_{i}-1]E_{\bf Q}[\Pi_{i-1}]+E_{\bf Q}[\psi_{i}-\phi_{i}]E_{\bf Q}[\Pi_i] \leq E_{\bf Q}[\Pi_i]\sup_{i\in \N}E_{\bf Q}[\psi_i]\leq {\widetilde C}_{\psi} \sum_{j=1}^{i} {\bf f}'(q)^j\beta^j,
\]
and the result follows immediately, since $\beta {\bf f}'(q)> 1$.
\end{proof}

Finally, we get the following
\begin{proposition}{\label{Sinfty}}
We have 
\[
E_{\bf Q}[S_{\infty}]\leq \frac{2 \widetilde{C}_{\psi}}{1-(\beta {\bf f}'(q))^{-1}}\left(\frac{\beta}{\beta-1}+  \frac 1 {1-{\bf f}'(q)}\right) <\infty.
\]
\end{proposition}
\begin{proof}
Recalling Lemma~\ref{L}, we get
\[
E_{\bf Q}[S_{\infty}] \leq  2\sum_{i=0}^{\infty} \beta^{-i} E_{\bf Q}[1+\Lambda_i]\leq  \frac{2 \widetilde C_{\psi}}{1-(\beta {\bf f}'(q))^{-1}}  \sum_{i=0}^{\infty} \beta^{-i} (1+ (\beta {\bf f}'(q))^i) <\infty,
\]
the last term being finite since $ {\bf f}'(q)<  1$.
\end{proof}

\section{Sketch of the proof}{\label{sketch}}

In the first step, we show (see Theorem~\ref{cuts}) that the time is essentially spent in $h_n$-traps.

Then we show that these $h_n$-traps are far away from each other, and thus the correlation between the time spent in different $h_n$-traps can be neglected. Moreover the number of $h_n$-traps met before level $n$ is roughly  $\rho C_a n^{\epsilon}$. 
Let
\begin{equation}
\label{chi0def}
\chi_0(n) = \text{ the time spent in the first } h_n -\text {trap met }
\end{equation}
where we point out that there can be several visits to this trap.
At this point we have reduced our problem to estimating
\[
\Delta_n \approx  \chi_1(n)+ \cdots + \chi_{\rho C_an^{\epsilon}}(n),
\]
where $\chi_i(n)$ are i.i.d.~copies of $\chi_0(n)$.

Now we decompose the time spent in the first $h_n$-trap according to the number of excursions in it starting from the $root$
\[
\chi_1(n)= \sum_{i=1}^{W_n} T_0^{(i)},
\]
where $W_n$ denotes the number of visits of the trap until time $n$ and $T_0^{(i)}$ an i.i.d.~sequence of random variables measuring the time spent during an excursion in a big trap. It is important to notice that the presence of an $h_n$-trap at a vertex gives information on the number of traps at this vertex, and thus on the geometry of the backbone. So the law of $W_n$ depends on $n$. Nevertheless we show that this dependence can be asymptotically neglected, and that for large $n$, $W_n$ is close to some random variable $W_{\infty}$ (Proposition~\ref{Winfty}).

Now we have essentially no more correlations between what happens on the backbone and on big traps. The only thing left to understand is the time spent during an excursion in an $h_n$-trap from the $root$. To simplify if the walker does not reach the point $\delta$ in the trap (this has probability $\approx 1- p_{\infty}$), the time in the trap can be neglected. Otherwise, the time spent to go to $\delta$, and to go directly from $\delta$ back to the $root$ of the trap can also be neglected, in other words, only the successive excursions from $\delta$ contribute to the time spent in the trap. This is developed in Section~\ref{bottom}, and we have 
\begin{equation}
\label{skapprox2}
\chi_1(n)\approx \sum_{i=1}^{\text{Bin}(W_{\infty},p_{\infty})} \sum_{j=0}^{G^{(i)}-1} T_{exc}^{(i,j)},
\end{equation}
where $T_{exc}^{(i,j)}$ are i.i.d.~random variables giving the lengths of the excursions from $\delta$ to $\delta$.
Further, $G^{(i)}$ is the number of excursions from $\delta$ during the $i$-th excursion in the trap: it is a geometric random variable with a parameter of order $\beta^{-H}$. Since $\beta^{-H}$ is very small ($H$ being conditioned to be big), the law of large numbers should imply that
\[
\sum_{j=0}^{G^{(i)}-1} T_{exc}^{(i,j)} \approx G^{(i)} E^{\omega}[T_{exc}^{(i,j)}] \approx G^{(i)} S_{\infty},
\]
and also we should have $G^{(i)}-1 \approx \beta^H {\bf e}_i$. This explains why, recalling (\ref{Zdef}),
\[
\chi_1(n)\approx  \beta^H \mathcal{Z}_{\infty}.
\]

We are then reduced to considering sums of i.i.d~random variables of the form $Z_i\beta^{X_i}$ with $X_i$ integer-valued. This is investigated in Section~\ref{sect_sum_iid}. We then finish the proof of Theorem \ref{subsequ} in Section 
\ref{lim-thm}.

\begin{remark}
The reasoning fails in the critical case $\gamma=1$, indeed in this case we have to consider a critical height $h_n$ which is smaller. This causes many problems, in particular in big traps there can be big subtraps and so, for example, the time to go from the top to the bottom of a trap cannot be neglected anymore. 
\end{remark}

\section{The time is essentially spent in big traps}\label{onlybig}

We recall that $h_n=\lceil -(1-\epsilon)\ln n/\ln {\bf f}'(q)\rceil$. Lemma \ref{tailH} gives the probability that a trap is an $h_n$-trap:  
\begin{equation}
\label{notationetan}
\eta_n:={\bf Q}[H\geq h_n] \sim \alpha {\bf f}'(q)^{h_n},
\end{equation}

For $x \in backbone$, we denote 
\begin{equation} 
\label{Ldef}
L_x \text{ the set of traps rooted at } x
\end{equation}
(if $x$ is not in the backbone then $L_x = \emptyset$). Let us denote the vertices in big traps by 
$L(h_n)=\{y \in T(\omega): y \text{ is in an } h_n -\text{ trap.}\}$.

Our aim in this section is to show the following
\begin{proposition}{\label{cuts}}
For $\epsilon>0$, we have 
\[
\text{for all $t\geq 0$,} \qquad \PR\left[\abs{\frac{\Delta_n-\chi(n)}{n^{1/\gamma}}} \geq t\right] \to 0
\]
where 
\begin{equation}\label{kidef}
\chi(n)=\card \{ 1 \leq i \leq \Delta_n : X_{i-1}, X_i \in L(h_n)\}
\end{equation}
is the time spent in big traps up to time $\Delta_n$.
\end{proposition}

Define
\begin{itemize}
\item[(i)] $A_1(n)=\{\Delta_n^Y \leq C_1 n \}$,
\item[(ii)]$A_2(n)=\{\card \cup_{i=1}^{\Delta_n^Y} L_{Y_i} \leq C_2 n \}$
\item[(iii)]$A_3(n)=\Bigl\{\displaystyle{\max_{\ell \in L_{Y_i}, i\leq \Delta_n^Y}} \card \{ 0 \leq i \leq \Delta_n^Y : Y_i \in \ell, \ X_{\sigma_i+1}\in \ell \}\leq C_3 \ln n\Bigr\}$,
\item[(iv)] $A(n)=A_1(n)\cap A_2(n) \cap A_3(n)$.
\end{itemize}

The following lemma tells us that typically the walk spends less than $C_1n$ time units before reaching level $n$, sees less than $C_2n$ traps and enters each trap at most $C_3 \ln n$ times.

\begin{lemma}\label{A1}
For appropriate constants $C_1$, $C_2$ and $C_3$, we have 
\[
\PR[A_1(n)^c]=o(n^{-2}) \text{ and } \PR[A(n)^c] \to 0.
\]
\end{lemma}

\begin{proof}
By a comparison to the $\beta$-biased random walk on $\Z$, standard large deviations estimates yields
\[
\PR[A_1(n)^c]=o(n^{-2}),
\]
for $C_1$ large enough.

On $A_1(n)$, the number of different vertices visited by $(Y_i)_{i\geq 0}$ up to time $\Delta_n^Y$ is at most $C_1 n$. The descendants at each new vertex are drawn independently of the preceding vertices. Moreover at each vertex the mean number of traps is at most the mean number of children, thus  ${\bf E}[\card L_0] \leq {\bf m}/(1-q)$. The law of large numbers yields for $C_2 > C_1 {\bf m}/(1-q)$ that 
\[
\PR[A_2(n)^c] \leq \PR\Bigl[\sum_{i=0}^{C_1 n} \card L_0^{(i)} > C_2 n\Bigr]+ \PR[A_1(n)^c] \to 0,
\]
where $\card L_0^{(i)}$ are i.i.d.~random variables with the law of $\card L_0$. This yields the second part.

For $A_3(n)$, we want, given a vertex $x$ in the backbone and any $\ell\in L_x$ to give an upper bound on the number of transitions from $x$ to $y$, where $y$ is the bud associated to $\ell$. Let $z$ be an offspring of $x$ in the backbone. Then, 
at each visit to $x$, either the walker does not visit $y$ or $z$, or it has probability $1/2$ to visit $y$ first (or $z$ first). Hence,
\begin{itemize}
\item[(i)] the number of transitions from $x$ to $y$ before reaching $z$ is dominated by a geometric random variable of parameter $1/2$,
\item[(ii)] the number of transitions from $x$ to $z$ is dominated by a geometric random variable of parameter $p_{\infty}$, since the escape probability from $z$ is at least $p_{\infty}$.
\end{itemize}

Consequently the number of transitions from $x$ to $y$ is dominated by a geometric random variable of parameter $p_{\infty}/2$. Thus
\[
\PR[A_3(n)^c\cap A_2(n)] \leq C_2 n \PR\Bigl[G(p_{\infty}/2)\geq C_3 \ln n\Bigr]\leq C n^{C_3\ln(1- p_{\infty}/2)+1},
\]
and if we take $C_3$ large enough we get the result.
\end{proof}

Now we can start proving Proposition~\ref{cuts}. Decompose $\Delta_n$ into
\begin{equation}
\label{decomposdelta}
\Delta_n = \Delta_n^Y+\chi(n) + \sum_{\ell \in \displaystyle{\cup_{i=0}^{\Delta_n^Y} L_{Y_i}\setminus L(h_n)}} N(\ell),
\end{equation}
where $N(\ell) = \card \{ 1 \leq i \leq \Delta_n : X_{i-1} \in \ell, X_i \in \ell\}$.

The distribution of $N(\ell)$ conditioned on the backbone, the buds and $(Y_i')_{i\leq \Delta_n^{Y'}}$, the walk on the backbone and the buds, is $\sum_{i=1}^{E_{\ell}} R_{\ell}^{(i)}$. Here we denoted $E_{\ell}$ the number of visits to $\ell$ and $R_{\ell}^{(i)}$ is the return time during the $i$-th excursion from the top of $\ell$. These quantities are considered for traps $\ell$, conditioned to have height at most $h_n$. 

Obviously we get from~(\ref{decomposdelta}) that
\begin{equation}
\label{theorem1_minoration}
\Delta_n\geq \chi(n).
\end{equation}

From~(\ref{decomposdelta}) we get for $t>0$,
\begin{equation}{\label{theorem1_decompo}}
\begin{array}{ll}
\PR\Bigl[\frac{\Delta_n-\chi(n)}{n^{1/\gamma}}>t\Bigr] & \leq \PR[A(n)^c] +\PR\Bigl[C_1n + \sum_{i=0}^{C_2 n}\sum_{j=0}^{C_3 \ln n} R_{\ell}^{(i)} \geq tn^{1/\gamma} \Bigr] \\
                                                      & \leq o(1) + \PR\Bigl[\sum_{i=0}^{C_2 n}\sum_{j=0}^{C_3 \ln n} R_{\ell}^{(i)} \geq \frac t 2  n^{1/\gamma}\Bigr],
\end{array}
\end{equation}
where we used Lemma~\ref{A1}.

Chebyshev's inequality yields,
\begin{align*}
\PR\Bigl[\sum_{i=0}^{C_2 n}\sum_{j=0}^{C_3 \ln n} R_{\ell}^{(i)} \geq \frac t 2 n^{1/\gamma} \Bigr] &\leq \frac 2 {tn^{1/\gamma}} \ES\Bigl[\sum_{i=0}^{C_2 n}\sum_{j=0}^{C_3 \ln n} R_{\ell}^{(i)}\Bigr]\leq \frac {2 C_2C_3 n^{1-1/\gamma} \ln n} t \ES[R_{1}^{(1)}]\, .
\end{align*}
Using Lemma~\ref{Z} and Lemma~\ref{meanreturntime}, we have
\begin{align*}  
\ES[R_{1}^{(1)}]&=E_{\bf Q}[E^{\omega}_{root}[T_{root}^+]|H< h_n] = 2\sum_{i=0}^{h_n-1} \beta^i E_{\bf Q}[Z_n|H< h_n]\\
 &\leq 2\sum_{i=0}^{h_n-1} (\beta {\bf f}'(q))^{i} \leq C n^{(1-\epsilon)(-1+ 1 /\gamma)}.
 \end{align*}

Plugging this in the previous inequality, we get for any $\epsilon>0$ and $t>0$ 
\[
\PR\Bigl[\sum_{i=0}^{C_2 n}\sum_{j=0}^{C_3 \ln n} R_{\ell}^{(i)} \geq \frac t 2 n^{1/\gamma}\Bigr]=o(1),
\]
thus recalling~(\ref{theorem1_decompo}) and~(\ref{theorem1_minoration}) we have proved Proposition~\ref{cuts}. \qed

\section{Number of visits to a big trap}{\label{secWinfty}}

We denote $K_x=\max_{\ell \in L_x} H(\ell)$, the height of the biggest trap rooted at $x$ for $x\in backbone$, where we recall that $H$ denotes the height of the trap from the bud and not from the $root$.

\begin{lemma}{\label{trapproba}}
We have
\[
{\bf P}[K_0 \geq h_n] \sim C_a {\bf f}'(q)^{h_n},
\]
where $C_a= \alpha q \frac{{\bf m}-{\bf f}'(q)}{1-q}$, recalling Lemma \ref{tailH} for the definition of $\alpha$.
\end{lemma}

\begin{proof}
We denote $Z$ the number of children of the root and $Z^*$ the number of children with an infinite line of descent. Let $P$ be the law of a ${\bf f}$-Galton Watson tree which is not conditioned on non-extinction and $E$ the corresponding expectation. Recall (\ref{notationetan}) and let $H^{(i)}, i= 1,2, \ldots $ be i.i.d.~random variables which have the law of the height of an ${\bf h}$-Galton-Watson tree, and are independent of $Z-Z^*$. Then
\[
{\bf P}[K_0 \geq h_n]={\bf P}\Bigl[\max_{i=1, \ldots ,Z-Z^*} H^{(i)} \geq h_n\Bigr] =1- \frac{E[(1-\eta_n)^{Z-Z^*} (1-\1{Z^*=0})]}{1-q},
\]
where the indicator function comes from the conditioning on non-extinction, which corresponds to $Z^*\neq 0$. 

Hence
\[
{\bf P}[K_0 \geq h_n]=1- \frac{E[(1-\eta_n)^{Z-Z^*}]-E[(1-\eta_n)^{Z-Z^*}0^{Z^*}]}{1-q},
\]
and using $E[s^{Z-Z^*}t^{Z^*}]={\bf f}(sq+t(1-q))$ (see~\cite{lycap}) we get
\[
{\bf P}[K_0 \geq h_n]=1-\frac{{\bf f}((1-\eta_n)q+1-q)-{\bf f}((1-\eta_n)q)}{1-q}.
\]

Now, using (\ref{notationetan}) and the expansion ${\bf f}(z-x)={\bf f}(z)-{\bf f}'(z)x+o(x)$ for $z\in \{q,1\}$, we get the result.
\end{proof}

Define the first time when we meet the root of an $h_n$-trap using the clock of $Y_n$,
\begin{equation}
\label{notationK}
K(n) = \inf \{i \geq 0 | K_{Y_{i}} \geq h_n\}.
\end{equation}

We also define ${\bf \ell}(n)$ to be an $h_n$-trap rooted at $Y_{K(n)}$, if there are several possibilities we choose one trap according to some predetermined order. We denote ${\bf b}(n)$ the associated bud.

We describe, on the event $0-\text{SR}$, the number of visits to $\ell (n)$, by the following random variable:
\begin{equation}
\label{notationW}
W_n=\card\{i: X_i=Y_{K(n)}, X_{i+1}={\bf b}(n)\},
\end{equation}
where $\omega$ is chosen under the law ${\bf P}[\cdot]$ and $X_n$ under $P_0^{\omega}[\cdot|0-\text{SR}]$. We will need the following bounds for the random variables $(W_n)_{n\geq 1}$.

\begin{lemma}
\label{tailW} We have $W_{n} \preceq G(p_{\infty}/3)$ for $n\in \N$, i.e. the random variables $W_n$ are stochastically dominated by a geometric random variable with parameter $p_{\infty}/3$.
\end{lemma}
\begin{proof}
For $n\in \N$, starting from any point $x$ of the backbone, the walker has probability at least $1/3$ to go to an offspring $y$ of $x$ on the backbone before going to ${\bf b}(n)$ or $\overleftarrow{x}$. But the first hitting time of $y$ has probability at least $p_{\infty}$ to be a super-regeneration time. The result follows as in the proof of Lemma \ref{A1}.
\end{proof}

\begin{proposition}{\label{Winfty}}
There exists a random variable $W_{\infty}$ such that
\[
W_n \xrightarrow{d} W_{\infty}.
\]
where we recall that for the law of $W_n$, $\omega$ is chosen under the law ${\bf P}[\cdot]$ and $X_n$ under $P_0^{\omega}[\cdot|0-\text{SR}]$.
\end{proposition}

\begin{remark}\label{Wdom}
It follows from Lemma~\ref{tailW} that $W_{\infty}\preceq G(p_{\infty}/3)$.
\end{remark}

Fix $n\in \N^*$ and set $m\geq n$. We aim at comparing the law of $W_m$ with that of $W_n$ and to do that we want to study the behaviour of the random walk starting from the last super-regeneration time before an $h_n$-trap (resp. $h_m$-trap) is seen. This motivates the definition of the last super-regeneration time seen before time $n$,
\[
\Sigma(n):=\max\{0\leq i \leq n: \ i-\text{SR}\}.
\]

For our purpose it is convenient to introduce a modified version of $W_m$, which will coincide with high probability with it. For $m\geq n$, recall that $\theta$ denotes the time-shift for the walk and set
\[
\overline{K}(m,n)=\inf \{j \geq 0: K_{Y_j} \geq h_m,\ {\bf \ell}(m)\circ \theta_{\Sigma(j)}={\bf \ell}(n) \circ \theta_{\Sigma(j)}\},
\]
the first time the walker meets a $h_m$-trap which is the first $h_n$-trap of the current regeneration block and we denote by $b(m,n)$ the associated bud. Set
\[
\overline{W}_{m,n}=\card\{i: X_i=Y_{\overline{K}(m,n)}, X_{i+1}=b(m,n)\},
\]
where $\omega$ is chosen under the law ${\bf P}[\cdot]$ and $(U_i)_{i\leq \overline{K}(m,n)}$ under $P_0^{\omega}[\cdot|0 \text{-SR}]$.

\begin{lemma}{\label{modif}}
For $m\geq n$ we have that 
\[
\overline{W}_{m,n}\stackrel{d}{=} W_n.
\]
\end{lemma}

\begin{proof}
To reach a vertex where an $h_m$-trap is rooted, the walker has to reach a vertex where an $h_n$-trap is rooted. Two cases can occur: either the first $h_n$-trap met is also a $h_m$-trap or it is not.
In the former case, which has probability $\eta_m/\eta_n>0$, since the height of the first $h_n$-trap met is independent of the sequence $(U_i)_{i \leq K(n)}$, the random variables $\overline{W}_{m,n}$ and $W_n$ coincide.
In the latter case, by its definition, $\overline{K}(m,n)$ cannot occur before the next super-regeneration time, hence $\overline{K}(m,n)\geq  \tau_1 \circ \theta_{K(n)}$. In this case $\overline{W}_{m,n}=\overline{W}_{m,n}\circ \theta_{K(n)}$ and then by Proposition~\ref{regenstruc1}, 
\[
\overline{W}_{m,n}\circ \theta_{\tau_1 \circ \theta_{K(n)}} \stackrel{d}{=} \overline{W}_{m,n},
\]
and $\overline{W}_{m,n}\circ \theta_{\tau_1 \circ \theta_{K(n)}}$ is independent of $(U_i)_{i\leq \tau_1 \circ \theta_{K(n)}-1}$.

The scenario repeats itself until the $h_n$-trap reached is in fact a $h_m$-trap, the number of attempts necessary to reach this $h_m$-trap is a geometric random variable of parameter $\eta_m/\eta_n$ which is independent of the $(U_i)$'s. 

This means that there is a family $(W_n^{(i)})_{i\geq 1}$ of i.i.d.~random variables with the same law as $W_n$ such that 
\[
\overline{W}_{m,n}=W_n^{(G)},
\]
where $G$ is a geometric random variable independent of the $(W_n^{(i)})_{i\geq 1}$. Then, note that we have
\[
\overline{W}_{m,n}=W_n^{(G)}\stackrel{d}{=}W_n.
\]
\end{proof}

Now we need to show that $\overline{W}_{m,n}$ and $W_m$ coincide with high probability, so we introduce the event
\[
A_{m,n}=\{{\bf \ell}(m)={\bf \ell}(n) \circ \theta_{\Sigma(K(m))} \},
\]
on which clearly $\overline{W}_{m,n}$ and $W_m$  are equal.

\begin{lemma}
\begin{equation}
\label{AMN1}
\sup_{m\geq n}\PR[A_{m,n}^c| 0-\text{SR}] \to 0 \quad\text{ for } n \to \infty .
\end{equation}
\end{lemma}

\begin{proof}
Let us denote, recalling (\ref{Ldef})
\[
V_j^i = \Bigl\{\card \bigcup\limits_{k=0}^{\tau_1} \{\ell\in L_{Y_k}, \ell \text{ is a $h_j$-trap}\} =i\Bigr\},
\]
and
\[
V_j^{i,+}=\Bigl\{\card \bigcup\limits_{k=0}^{\tau_1} \{\ell\in L_{Y_k}, \ell \text{ is a $h_j$-trap}\} \geq i\Bigr\}.
\]

Then we have 
\begin{equation}
\label{proba1}
\PR[A_{m,n}^c| 0-\text{SR}] \leq \PR[V_n^{2,+} | V_m^{1,+}, 0-\text{SR}].
\end{equation}

Let us denote $ \card \text{Trap} $ the number of traps seen before $\tau_1$,
\[
\card \text{Trap} = \card \Bigl\{ \ell: \ell \in \bigcup\limits_{i=0}^{\tau_1} L_{Y_i}\Bigr\},
\]
and its generating function by 
\[
\varphi(s):=\ES\Bigl[s^{\card \text{Trap}}|0-\text{SR}\Bigr].
\]

The probability of $A_{m,n}$ can be estimated with the following lemma, whose proof is deferred.

\begin{lemma}{\label{uniqueness}}
We have
\[
\forall m\geq n,\ \PR[V_n^{2,+} | V_m^{1,+}, 0-\text{SR}] \leq \varphi'(1)-\varphi'(1-\eta_n).
\]
\end{lemma}

Now we have $\ES\Bigl[\card \text{Trap} | 0-\text{SR}\Bigr]\leq \ES[\tau_1|0-\text{SR}]\ES[\card L_0]<\infty$ because of Remark~\ref{avantregen} and hence $\varphi'$ is continuous at $1$, and (\ref {AMN1}) follows from (\ref{proba1}).
\end{proof}

Applying Lemma~\ref{uniqueness}, Lemma~\ref{modif} and~(\ref{AMN1}) we get,
\begin{align*}
\PR[W_m \geq y | 0-\text{SR}]&=\PR[A_{m,n},\overline{W}_{m,n} \geq y| 0-\text{SR}] +o(m,n)\\  
                            &=\PR[\overline{W}_{m,n} \geq y| 0-\text{SR}]+o(m,n)\\
                            &=\PR[W_n \geq y | 0-\text{SR}]+o(m,n),
\end{align*}
where $\sup_{m\geq n} o(m,n)\to 0$ as $n$ goes to infinity.

The law of a random variable $W_{\infty}$ can be defined as a limit of the laws of some subsequence of $(W_m)$, 
since the family $(W_m)_{m\geq 0}$ is tight by Lemma~\ref{tailW}. Then taking $m$ to infinity along this subsequence in the preceding equation yields
\[
\forall t>0,\ \PR[W_{\infty} \leq t\mid 0-\text{SR} ] = \PR[W_n \leq t |0-\text{SR}]+o(1).
\]
This proves Proposition~\ref{Winfty}. \qed

It remains to show Lemma \ref{uniqueness}.  
\begin{proof}
Note that for $i \geq 1$,
\begin{align}
\label{bayes}
\PR[V_n^{i} | V_m^{1,+}, 0-\text{SR}] & = \frac{\PR[V_n^{i}| 0-\text{SR}]}{\PR[V_m^{1,+}| 0-\text{SR}]} \PR[V_m^{1,+}| V_n^{i}, 0-\text{SR}] \\ \nonumber
                        & \leq \frac {\PR[V_n^{i}| 0-\text{SR}]} {\PR[V_m^{1,+}| 0-\text{SR}]} i  {{\bf Q}[H\geq h_m| H \geq h_n]} \\ \nonumber
                        &= i \frac{\PR[V_n^{i}| 0-\text{SR}]}{\eta_n} \frac{\eta_m}{\PR[V_m^{1,+}| 0-\text{SR}]} \\ \nonumber
                        &\leq i \frac{\PR[V_n^{i}| 0-\text{SR}]}{\eta_n}.
\end{align}

Then we have
\begin{align*}
                                          &\sum_{i\geq 2} i \PR[V_n^{i}| 0-\text{SR}] \\
                                       =  &  \sum_{i \geq 2} \sum_{j \geq i} \PR[\card \text{Trap} =j| 0-\text{SR}] i \binom{j}{i}\eta_n^i(1-\eta_n)^{j-i}\\
                                       =  &\sum_{j \geq 0} j \PR[\card \text{Trap} =j| 0-\text{SR}]  \sum_{i = 2}^j \binom{j-1}{i-1}\eta_n^i(1-\eta_n)^{j-i} \\
                                        =  & \eta_n  \sum_{j \geq 0} j \PR[\card \text{Trap} =j| 0-\text{SR}]  \sum_{i = 1}^{j-1} \binom{j-1}{i}\eta_n^i(1-\eta_n)^{(j-1)-i} \\
                                        = & \eta_n \sum_{j \geq 0} j \PR[\card \text{Trap} = j| 0-\text{SR}] \Bigl(1-(1-\eta_n)^{j-1}\Bigr) \\
                                        = & \eta_n (\varphi'(1)-\varphi'(1-\eta_n)) .                                      
\end{align*}

Inserting this in~(\ref{bayes}) we get
\begin{align*}
\PR[V_n^{2,+}| V_m^{1,+}, 0-\text{SR}]&=\sum_{i=2}^{\infty} \PR[V_n^{i}| V_m^{1,+}, 0-\text{SR}] \\
                                     &\leq  \frac 1 {\eta_n} \sum_{i\geq 2} i \PR[V_n^{i}| 0-\text{SR}]= \varphi'(1)-\varphi'(1-\eta_n),
\end{align*}
which concludes the proof of Lemma~\ref{uniqueness}.
\end{proof}

We will need the following lower bound for the random variable $W_\infty$.
\begin{lemma}
\label{minorWinfty}
There exists a constant $c_W>0$ depending only on $(p_i)_{i\geq 0}$, such that
\[
\PR[W_{\infty}\geq 1] \geq c_W.
\]
\end{lemma}

\begin{proof}
By Proposition~\ref{Winfty}, it is enough to show the lower bound for all $W_n$. First let us notice that
\begin{equation}
\label{firststepminorWinfty}
\PR[W_{n}\geq 1\mid 0-\text{SR}] \geq \ES\left[\left(\frac 1 {Z(K(n))+1}\right)^2 p_{\infty}\right]\geq 
\left(1-{\bf f}'(q)\right)\ES\Bigl[\frac{1}{\left(Z(K(n))+1\right)^2}\Bigl],
\end{equation}
where $Z(K(n))$ is the number of offspring of $Y_{K(n)}$. To show (\ref{firststepminorWinfty}), note that the particle has probability at least 
$\beta/(\beta Z(K(n))+1)\geq 1/(Z(K(n))+1)$ of going from $Y_{K(n)}$ to ${\bf b}(n)$ and when it comes back to $Y_{K(n)}$ again there is probability at least $1/(Z(K(n))+1)$ to go from $Y_{K(n)}$ to one of its descendants on the backbone and then there is a probability of at least $p_{\infty}$ that a super-regeneration occurs. The event we just described is in $\{W_n\geq 1\}\cap \{0-\text{SR}\}$. For the second inequality in (\ref{firststepminorWinfty}), use $\beta > \beta_c = {\bf f}'(q)^{-1}$ hence $p_\infty = 1-\beta^{-1} \geq 1-{\bf f}'(q)$.

Now, we notice that the law of the $Z(K(n))$ is that of $Z_1$ conditioned on the event $\{\text{an } h_n-\text{trap is rooted at } 0\}$. Denote $j_0$ the smallest index such that $j_0>1$ and $p_{j_0}>0$ (which exists since ${\bf m}>1$) and $Z_1^*$ the number of descendants of $0$ with an infinite line of descent. All $Z_1-Z_1^*$ traps rooted at  $0$ have, independently of each other, probability $\eta_n$ of being $h_n$-traps, so that
\begin{align*}
\PR[Z(K(n))=j_0] & ={\bf P}[Z_1=j_0 \mid \text{an } h_n-\text{trap} \text{ is rooted at } 0] \\
                & =\frac{{\bf P}[Z_1=j_0,\ \text{Bin}(Z_1-Z_1^*,\eta_n) \geq 1] }{ {\bf P}[\text{an } h_n-\text{trap} \text{ is rooted at } 0]}\\
                & \geq \frac{\eta_n{\bf P}[Z_1=j_0,Z_1-Z_1^*\geq 1] }{ {\bf P}[\text{an } h_n-\text{trap} \text{ is rooted at } 0]}.
\end{align*}
Further,  since $P[\text{Bin}(Z_1-Z_1^*,\eta_n) \geq 1] \leq Z_1 \eta_n$, we have
\[
{\bf P}[\text{an } h_n-\text{trap} \text{ is rooted at } 0]\leq \sum_{j=0}^{\infty} {\bf P}[Z_1-Z_1^*=j] j \eta_n \leq {\bf m} \eta_n.
\]

Putting these equations together, we get that
\[
\PR[Z(K(n))=j_0] \geq \frac{{\bf P}[Z_1=j_0,Z_1-Z_1^*\geq 1]}{{\bf m}}.
\]

The last equation and~(\ref{firststepminorWinfty}) yield a lower bound for $\PR[W_{\infty}\geq 1]$ which depends only on $(p_k)_{k\geq 0}$.
\end{proof}

\section{The time spent in different traps is asymptotically independent}{\label{indep}}

In order to show the asymptotic independence of the time spent in different big traps we shall use super-regeneration times. First we show that 
the probability that there is an $h_n$-trap in the first super-regeneration block goes to $0$ for $n \to \infty$.

Define 
\begin{itemize}
\item[(i)]$B_1(n)=\{\forall i\in [1,n],\ \card \{Y_{\tau_i},\ldots ,Y_{\tau_{i+1}} \}
\leq n^{\epsilon}\}$,
\item[(ii)]$B_2(n)= \{ \forall i \in [0, \tau_1], \ \card L_{Y_i} \leq n^{2\epsilon} \}$,
\item[(iii)]$B_3(n)=\{ \forall i \in [0,\tau_1],\ \forall \ell\in L_{Y_i},\ \ell \text{
is not an $h_n$-trap} \}$,
\item[(iv)]$B_4(n)=\{ \forall i\in [2,n],\ \card \bigl\{j:
j\in[\tau_i,\tau_{i+1}], \  L_{Y_j}   \text{ contains an $h_n$-trap}\bigr\} \leq 1\}$,
\item[(v)]$B(n)=B_1(n)\cap B_2(n)\cap B_3(n)\cap B_4(n)$.
\end{itemize}

\begin{lemma}{\label{B1}}
For $\epsilon<1/4$, we have
\[
\PR[B_1(n)^c]=o(n^{-2}) \text{ and } \PR[B(n)^c] \to 0.
\]
\end{lemma}

\begin{proof}
Since $\tau_2- \tau_1$ (resp. $\tau_1$) has some positive exponential moments and
$B_1(n)^c \subseteq \cup_{i=1}^n\{\tau_{i+1} - \tau_i \geq n^{\epsilon}\}$,
\[
\PR[B_1(n)^c] =o(n^{-2}).
\]

Using the fact that the number of traps at different vertices has the same law,
\[
\PR[B_2(n)^c] \leq \PR[B_1(n)^c]+n^{\epsilon} {\bf P}[ \card L_0 \geq n^{2\epsilon}] \leq
o(1)+n^{-\epsilon} \frac{{\bf m}}{1-q}=o(1),
\]
where we used Chebyshev's inequality and ${\bf E}[\card L_0] \leq {\bf E}[Z_1] \leq {\bf
m}/(1-q)$.

Then we have
\[
\PR[B_3(n)^c]\leq \PR[B_2(n)^c] + n^{3\epsilon} \eta_n =o(1),
\]
yielding the result using~(\ref{notationetan}),  since $\epsilon< 1/4$.

Finally, up to time $n$ we have at most $n$ super-regeneration blocks, on $B_1(n)$ they
contain at most $n^{\epsilon}$ visited vertices. But the probability that among the
$n^{\epsilon}$ first visited vertices after a super-regeneration time, two of them are adjacent to a
big trap is bounded above by $n^{2\epsilon}{\bf P}[K_0\geq h_n]^2$ (here we implicitly used Remark \ref{avantregen}). Hence, we get
\[
\PR[B_4(n)^c] \leq \PR[B_1(n)^c]+n n^{2\epsilon}(C n^{\epsilon-1})^2 =O(n^{4\epsilon-1}),
\]
yielding the result for $\epsilon<1/4$.
\end{proof}

We define $R(n)=\card\{Y_1,\ldots,Y_{\Delta_n^Y}\}$ and 
$l_n$ the number of vertices where an $h_n$-trap is rooted: 
\begin{equation}\label{ellendef}
l_n=\card\Bigl\{ i \in [0,\Delta_n^Y]:  L_{Y_i} \text{ contains an $h_n$-trap}\Bigr\}.
\end{equation}
Recall~(\ref{notationrho}) and define
\begin{equation}
\label{C1def}
C_1(n)=\{ (1-n^{-1/4})\rho n \leq R(n) \leq (1+n^{-1/4})\rho n\}
\end{equation}
\begin{equation} 
\label{C2def}
\displaystyle{C_2(n)=\Bigl\{ (1-n^{-\epsilon/4})\rho C_a  n{\bf f}'(q)^{h_n} \leq l_n \leq (1+n^{-\epsilon/4})\rho C_a n{\bf f}'(q)^{h_n}\Bigr\}}
\end{equation}
\begin{equation}
\label{C3def}
C_3(n)=\displaystyle{\Bigl\{ \forall 1\leq i\leq \Delta_n^Y ,\ \card\Bigl\{ \ell\in L_{Y_i} : \ell\text{ is an $h_n$-trap}\Bigr\}\leq 1\Bigr\}}
\end{equation}
and $C(n)=C_1(n) \cap C_2(n) \cap C_3(n)$.
\begin{lemma}{\label{C1}}
For $\epsilon <1/4$, we have
\[
\PR[C(n)^c] \to 0.
\]
\end{lemma}

\begin{proof}
First, we notice that for $i\geq 0$ and with the convention $\tau_0:=0$ we have
\[
Z_i : = \card\{Y_{\tau_i+1},\ldots,Y_{\tau_{i+1}}\}\preceq G^{(i)}(p_{\infty}),
\]
where the geometric random variables $G^{(i)}$ are i.i.d. Indeed, at each new vertex visited we have probability at least $p_{\infty}$ to have a super-regeneration time. Let us denote by $n_0$ the smallest integer such that $\Delta_n^Y \leq \tau_{n_0}$, which satisfies $n_0\leq n$ since $\abs{X_{\tau_{i+1}}}-\abs{X_{\tau_i}}\geq 1$. \\
Now, since the random variables $Z_i$ are i.i.d. and
$\sum\limits_{i=1}^{n_0 -1} Z_i \leq \card \{ Y_1, \ldots, Y_{\Delta_n^Y}\} \leq \sum\limits_{i=1}^{n_0} Z_i $,
we have
\begin{align*}
&\PR\Bigl[\Bigl\vert\frac{\card \{ Y_1, \ldots, Y_{\Delta_n^Y}\}} n - \rho\Bigr\vert\geq n^{-1/4}\Bigr] \\
\leq & n^{1/2} \text{Var}\Bigl(\frac{\card \{ Y_1, \ldots, Y_{\Delta_n^Y}\}} n\Bigr) \\
= & n^{1/2}\left( \ES\Bigl[\Bigl(\frac{\card \{ Y_1, \ldots, Y_{\Delta_n^Y}\}} n\Bigr)^2\Bigr]-\ES\Bigl[\frac{\card \{ Y_1, \ldots, Y_{\Delta_n^Y}\}} n\Bigr]^2 \right)\\
\leq & n^{-3/2} \left( \ES\Bigl[\Bigl(\sum\limits_{i=1}^{n_0} Z_i\Bigr)^2 \Bigr]-\ES\Bigl[\sum\limits_{i=1}^{n_0-1} Z_i \Bigr]^2\right) \\
\leq & n^{-1/2} (E[G(p_\infty)^2]+E[G(p_{\infty})]^2),
\end{align*}
yielding $\PR[C_1(n)^c] \to 0$.

On $C_1(n)$ we know that there are $R(n)\in[\rho n (1-n^{-1/4}) ,\rho n (1+n^{-1/4})]$ vertices where we have independent trials to have $h_n$-traps.
Hence $l_n$ has the law $\text{Bin}(R(n),{\bf P}[K_0\geq h_n] ) $, where the success probability satisfies ${\bf P}[K_0\geq h_n] \leq C n^{\epsilon-1}$ has asymptotics given by Lemma~\ref{trapproba}. Now, standard estimates for Binomial distributions imply that $\PR[C_2(n)^c \cap C_1(n)] \to 0$.

%

On $C_2(n)$, there are at most $C n^{\epsilon}$ vertices where (at least) one $h_n$-trap can be rooted, we only need to prove that, with probability going to $1$, those vertices do not contain more than two $h_n$-traps. Using the same reasoning as in Lemma~\ref{uniqueness} we get
\begin{align*}
& {\bf P}[\text{0 has at least two $h_n$-traps}|\text{0 has at least one $h_n$-trap}, 0-\text{ SR}]\\& \leq {\bf f}'(1)-{\bf f}'(1-\eta_n) \leq C \eta_n,
\end{align*}
where we used that ${\bf E}[Z^2]<\infty$, which implies that ${\bf f}''(1)<\infty$.

The result follows from the fact that $\eta_n=o(n^{-\epsilon})$ for $\epsilon<1/4$.
\end{proof}

Let us denote, recalling (\ref{notationH}),
\[
D(n)=\Bigl\{ \max_{\displaystyle{\ell\in \cup_{i=0, \ldots, \Delta_n^Y}L_{Y_i}}} H(\ell) \leq \frac {2\ln n} {- \ln {\bf f}'(q)}\Bigr\}.
\]

\begin{lemma}
\label{D1}
We have 
\[
\PR[D(n)^c] \to 0.
\]
\end{lemma}

\begin{proof}
Due to~(\ref{tailmaj}), we know that ${\bf Q}[H \geq\frac {2\ln n} {-\ln {\bf f}'(q)}] \leq n^{-2}$, so using Lemma~\ref{A1}
\[
\PR[D(n)^c] \leq \PR[A_2(n)^c]+\PR[A_2(n)\cap D(n)^c] \leq o(1)+C_2 n^{-1}=o(1),
\]
which concludes the proof.
\end{proof}

On $B(n)$ there is no big trap in the first super-regeneration block, on $B(n)\cap C(n)$ all big traps are met in distinct super-regeneration blocks and $C_2(n)$ tells us the asymptotic number of such blocks. Moreover on $D(n)$, we know that to cross level $n$ on a trap, it has to be rooted after level $n-(- 2 \ln n/\ln {\bf f}'(q))$. Hence using Lemma~\ref{B1}, Lemma~\ref{C1}, Lemma~\ref{D1}, Proposition~\ref{regenstruc1} and Remark~\ref{regenstruc2}, we get

\begin{proposition}
\label{approx_sum_iid}
Let $\chi_i(n), i\geq 1$, be i.i.d.~copies of $\chi_0(n)$, see (\ref{chi0def}), and \\
${\widetilde n} =n-(- 2 \ln n/\ln {\bf f}'(q))$. Then we have
\[
o_1(n) + \sum_{i=1}^{ (1-{\widetilde n}^{-\epsilon/4})\rho_n C_a  {\widetilde n}{\bf f}'(q)^{h_{{\widetilde n}}}}\chi_i(n) \leq \chi(n) \leq \sum_{i=1}^{ (1+n^{-\epsilon/4})\rho_n C_a  n{\bf f}'(q)^{h_n}}\chi_i(n) + o_2(n),
\]
where $\lim_{n \to \infty}o_1(n) =\lim_{n\to \infty} o_2(n) =0$.
\end{proposition}

In the light of Proposition~\ref{cuts}, our problem reduces to understanding the convergence in law of a sum of i.i.d.~random variables. The aim of the next section is to reduce $\chi_1(n)$ to a specific type of random variable for which limit laws can be derived (see Section~\ref{sect_sum_iid}).

\section{The time is spent at the bottom of the traps}{\label{bottom}}
We denote by $\ell_j(n)$ the 
$j$-th $h_n$-trap seen by the walker, and by
$\delta_j(n)$ (resp. $root_j(n)$, ${\bf b}_j(n)$) the leftmost bottom point (the root, the bud) of  
$\ell_j(n)$. Further, $\chi_j(n)$ denotes the time spent in the $j$-th $h_n$-trap met. Then, $\chi_i(n), i\geq 1$, are i.i.d. copies of $\chi_0(n)$ as in Proposition \ref{approx_sum_iid}.

We want to show that the time spent in the big traps is essentially spent at the bottom of them, i.e.~during excursions from the the bottom leftmost 
point $\delta$. In order to prove our claim, we introduce
\[
\chi_j^*(n)=\card \{k \geq 0 : \ X_k \in \ell_j(n), k\geq T_{\delta_j(n)}, T_{\delta_j(n)}\circ \theta_k<\infty\},
\]
the time spent during excursions from the bottom in the $j$-th $h_n$-trap met. It is obvious that  
\[
\chi_j(n) \geq \chi^*_j(n).
\]

We prove that 
\begin{proposition}
\label{th_neglectall}
For $\epsilon<1/4$, we have, recalling (\ref{ellendef})
\[
\text{for all $t>0$,} \qquad \PR\left[ \frac 1{n^{1/\gamma}}\Bigl\vert\sum_{j=1}^{l_n} \left( \chi_j(n)-\chi_j^*(n)\right) \Bigr\vert \geq t \right] \to 0.
\]
\end{proposition}

In order to prove the preceding proposition, we mainly need to understand $\chi_1(n)$ and $\chi_1^*(n)$. Note that $\chi_1(n)$ is a sum of $W_n$ successive i.i.d.~times spent in $\ell_1(n)$ and $\chi_1^*(n)$ is a sum of $W_n$ successive i.i.d.~times spent during excursions from the bottom of $\ell_1(n)$. We can rewrite the proposition as follows
\begin{equation}
\label{goal_th_neglectall}
\text{for all $t>0$,} \qquad \PR\Bigl[ \frac 1 {n^{1/\gamma}} \Bigl\vert\sum_{i=1}^{l_n} 
\sum_{j=1}^{W_n^{(i)}} (T_{root_i(n)}^j-T_{root_i(n)}^{*,j})\Bigr\vert \geq t \Bigr] \to 0
\end{equation}
where 
\[
T_{root_i(n)}^j=\card \{k\geq 0: \ X_k\in \ell_i(n),\ \card\{\widetilde{k}\leq k:\ X_{{\widetilde k}+1}={\bf b}_i(n),\ X_{\widetilde k}=root_i(n)\}=j\},
\]
and
\begin{align*}
T_{root_i(n)}^{*,j} =\card \{ k \geq 0: \ X_k\in \ell_i(n),\ \card\{{\widetilde k} \leq k: \ X_{{\widetilde k} +1} ={\bf b}_i(n),\ X_{\widetilde k} =root_i(n)\}=j, \\
k\geq T_{\delta_j(n)}, T_{\delta_j(n)}\circ \theta_k<\infty\},
\end{align*}
and $(W_n^{(i)}), i \geq 1$ are i.i.d.~copies of $W_n$.

Consequently, in this section we mainly investigate the walk on a big trap, which is a random walk in a finite random environment. Recall that $root$ is the vertex $Y_{K(n)}$ on the backbone where ${\bf \ell}(n)$ is attached. Moreover set ${\bf Q}_n[ \cdot ] = {\bf Q}[\cdot \mid H \geq h_n]$, 
$E_{{\bf Q}_n}[ \cdot] = E_{\bf Q}[ \cdot \mid   H \geq h_n]$, $E^{\omega}[\cdot] :=E_{root}^{\omega}[\cdot ]$ and $\ES_{{\bf Q}_n}[\cdot] = E_{{\bf Q}_n}[E^{\omega}[\cdot]]$.

\begin{remark}{\label{switchnotation}}
To ease notations, we add to all these probability spaces an independent random variable $W_n$ whose law is given by (\ref{notationW}),  under the law $\PR[\cdot|0-\text{SR}]$ for $n\in \N \cup \{ \infty \}$.
\end{remark}

We will extensively use the description of Section~\ref{trapcons}, in particular we recall that a trap is composed of $root$ which is linked by an edge to an ${\bf h}$-Galton-Watson tree.

We want to specify what ${\bf \ell}(n)$ looks like. Denoting 
\[
h_n^+=\Bigl\lceil \frac{(1+\epsilon) \ln n}{-\ln {\bf f}'(q)}\Bigr\rceil,
\]
consider
\begin{itemize}
\item[(i)]$\widetilde A_1(n)=\{H\leq h_n^+\}$,
\item[(ii)] $\widetilde A_2(n)=\{\text{there are fewer than $n^{\epsilon}$ subtraps}\}$, 
\item[(iii)]$\widetilde A_3(n)=\{\text{all subtraps of ${\bf \ell}(n)$ have  height $\leq h_n$}\}$,
\item[(iv)] $\widetilde A(n)=\widetilde A_1(n) \cap \widetilde A_2(n) \cap \widetilde A_3(n)$.
\end{itemize}

\begin{lemma}{\label{negA}}
For $\epsilon<1/4$, we have 
\[
{\bf Q}_n[\widetilde A(n)^c]=o(n^{-\epsilon}).
\]
\end{lemma}

\begin{proof}
First 
\[
{\bf Q}_n[\widetilde A_1(n)^c]\leq \frac{{\bf Q} [H\geq h_n^+]}{{\bf Q}[H\geq h_n]} \leq C n^{-2\epsilon} = o(n^{-\epsilon}).
\]

Furthermore using Lemma~\ref{W}, Lemma~\ref{Z} and~(\ref{tailmaj}), we get
\begin{align*}
    & {\bf Q}_n[\widetilde A_2(n)^c] \\
\leq & {\bf Q}_n[\widetilde A_1(n)^c] + {\bf Q}_n[\widetilde A_1(n), \text{there are $n^{\epsilon}/h_n^+$ subtraps on a vertex of the spine}] \\
\leq & o(n^{-\epsilon}) +h_n^+ C_{\psi}\frac{n^{\epsilon}}{h_n^+}q^{n^{\epsilon}/h_n^+} =o(n^{-\epsilon}).
\end{align*}

Finally 
\begin{align*}
    & {\bf Q}_n[\widetilde A_3(n)^c]\\
\leq & {\bf Q}_n[\widetilde A_2(n)^c] + {\bf Q}_n[\widetilde A_2(n), \text{ there exists a subtrap of height $\geq h_n$}] \\
\leq & o(n^{-\epsilon}) + n^{\epsilon} \eta_n =o(n^{-\epsilon}),        
\end{align*}
where we used (\ref{notationetan}) and $\epsilon<1/4$.    
\end{proof}

Using Chebyshev's inequality we get, recalling (\ref{C2def}),
\begin{align*}
&\PR\Bigl[ \frac 1 {n^{1/\gamma}}\Bigl\vert\sum_{i=1}^{l_n} \sum_{j=1}^{W_n^{(i)}} (T_{root_i(n)}^j-T_{root_i(n)}^{*,j})\Bigr\vert \geq t \Bigr]\\
                                 & \leq \frac 1 {tn^{1/ \gamma}}\ES\Bigl[\1{C_2(n)}\1{\widetilde A(n)}\sum_{i=1}^{l(n)} \sum_{j=1}^{W_n^{(i)}} (T_{root_i(n)}^j-T_{root_i(n)}^{*,j})\Bigr] + \PR[C_2(n)^c] +{\bf Q}_n[\widetilde A(n)^c] \\
                                 & \leq \PR[C_2(n)^c] + {\bf Q}_n[\widetilde A(n)^c]+\frac {2\rho C_a c_\beta n^{\epsilon}} {tn^{1/\gamma}} \ES[\1{\widetilde A(n)} (T_{root_1(n)}^1-T_{root_1(n)}^{*,1})],
\end{align*}

where $c_\beta =  E[G(p_{\infty}/3)]$, implying $\ES[W_n^{(j)}] \leq c_\beta$. Hence using Lemma~\ref{negA} and Lemma~\ref{C1},
with
$$
a_n : = \PR\Bigl[ \frac 1 {n^{1/\gamma}} \Bigl\vert\sum_{i=1}^{l_n} \sum_{j=1}^{W_n^{(i)}} (T_{root_i(n)}^j-T_{root_i(n)}^{*,j})\Bigr\vert \geq t \Bigr] 
$$
and 
$$
b_n : = 
\frac C t n^{\epsilon-1/\gamma}E_{{\bf Q}_n}[\1{\widetilde A(n)} (T_{root_1(n)}^1-T_{root_1(n)}^{*,1})]\, ,
$$
we have
\begin{equation}
\label{goal2_th_neglectall}
\limsup\limits_{ n\to \infty} \left(a_n - b_n\right) \leq 0\, .
\end{equation}

We have to estimate this last expectation. Consider an $h_n$-trap. Each time the walker enters the $h_n$-trap two cases can occur: either the walker will reach $\delta$, or he will not reach $\delta$ before he comes back to $root$. In the former case, $T_{root}^+-T_{root}^{*,+}$ is the time spent going from $root$ to 
$\delta$ for the first time plus the time coming back from $\delta$ to $root$ for the last time (starting from $\delta$ and going back to $root$ without returning to $\delta$). In the latter case, $T_{root}^+-T_{root}^{*,+}$ equals $T_{root}^+$. This yields the following upper bound
\begin{align}
\label{needneglectall}
& \ES[\1{\widetilde A(n)}(T_{root}^+-T_{root}^{*,+})]\\
\leq & E_{{\bf Q}_n}[\1{\widetilde A(n)}E_{root}^{\omega}[\1{\widetilde A(n)}T_{\delta}^+\mid T_{\delta}^+<T_{root}^+]]+  E_{{\bf Q}_n}[\1{\widetilde A(n)}E_{\delta}^{\omega}[T_{root}^+\mid T_{root}^+<T_{\delta}^+]] \nonumber \\
                       \nonumber & \qquad \qquad+  E_{{\bf Q}_n}[\1{\widetilde A(n)}E_{root}^{\omega}[T_{root}^+\mid T_{root}^+<T_{\delta}^+]].
\end{align}

To tackle the conditionings that appear, we shall use $h-$processes, see~\cite{ESZ2} and~\cite{Zeitouni} for further references. For a given environment $\omega$ let us denote $h^{\omega}$ the following function on the trap: $h^{\omega}(z)=P_z^{\omega}[T_{root}^+ < T_{\delta}^+]$,  with $h^{\omega}(root)=1$ and $h^{\omega}(\delta)=0$. Then we have the following formula for the transition probabilities of the conditioned Markov chain

\begin{equation}{\label{hproc}}P^{\omega}_y[X_1=z|T_{root}^+ < T_{\delta}]=\frac{h^{\omega}(z)}{h^{\omega}(y)}P^{\omega}_y[X_1=z],\end{equation}
for $y,z$ in the trap. 
Due to the Markov property, $h^{\omega}$ is harmonic except on $\delta$ and $root$. It can be computed easily here (note that this is a 
one-dimensional calculation involving only the spine):
\[
h^{\omega}(y)=h^{\omega}(y\wedge \delta)=\beta^{-d(root,y\wedge \delta)}\frac{1-\beta^{-(H+1-d(root,y\wedge \delta))}}{1-\beta^{-(H+1)}}.
\]

In particular, comparing the walk conditioned on the event $\{T_{\delta}^+>T_{root}^+\}$ to the original walk, we have the following:
\begin{enumerate}
\item the walk remains unchanged on the subtraps,
\item for $y$ on the spine and $z$ a descendant of $y$ not on the spine, we have 
\[
P^{\omega}_y[X_1=\overleftarrow{y} |T_{root}^+< T_{\delta}]>P^{\omega}_y[X_1 =z |T_{root}^+ < T_{\delta}],
\]
\item for $y\notin \{\delta,root\}$ on the spine, we have 
\[
P^{\omega}_y[X_1=\overleftarrow{y} |T_{root}^+ < T_{\delta}]>\beta P^{\omega}_y[X_1=\overrightarrow{y} |T_{root}^+ < T_{\delta}].\]
\end{enumerate}

The points $(2)$ and $(3)$ state respectively that the conditioned walk is more likely to go towards $root$ than to go to a given vertex of a subtrap and that restricted to the spine the conditioned walk is more than $\beta$-drifted towards $root$.

\begin{lemma}{\label{neglectup}}
For $z \in \{\delta,root\}$, we have
\[
E_{{\bf Q}_n} [\1{\widetilde A(n)}E_{z}^{\omega}[T_{root}^+\mid T_{root}^+<T_{\delta}^+]]\leq  C (\ln n) n^{(1-\epsilon)(1/\gamma -1)+\epsilon}.
\]
\end{lemma}

\begin{proof}
First let us show that the walk cannot visit too often a vertex of the spine. Indeed let $y$ be a vertex of the spine, using fact $(3)$, we have 
$P^{\omega}_y[T_y^+>T_{root}^+|T_{root}^+< T_{\delta}] \geq p_{\infty}$. Hence the random variable $N(y)=\card \{ n\leq T^+_{root}:\ X_n=y \}$ with $(X_n)$ conditioned on $\{T_{root}^+ < T_{\delta}\}$ is stochastically dominated by $G(p_{\infty})$, a geometric random variable with parameter $p_\infty$.

Furthermore, we cannot visit often a given subtrap $s(y)\in S_y$ (recall (\ref{Sdef})). Indeed, if we denote the number of visits to $s(y)$ by $N(s(y))=\card \{ n \leq T^+_{root}:\ X_n=y,\ X_{n+1} \in s(y) \}$, using Fact $(2)$ and a reasoning similar to the one for the asymptotics on $A_3(n)$ in Lemma~\ref{A1} we have that $N(s(y))$ with $(X_n)$ conditioned on $\{T_{root}^+< T_{\delta}\}$ is stochastically dominated by $G(p_{\infty}/2)$.

Let us now consider the following decomposition
\[
T_{root}^+= T_{spine}+ \sum_{s \in \text{subtraps}} \sum_{j=1}^{N(s)} R_{s}^{j},
\]
where $T_{spine}=\card\{n\leq T^+_{root} : X_n \text{ is in the spine}\}=\sum_{x\in \text{spine}} N(x)$, and $R_s^{j}$ is the time spent in the subtrap $s$ during the $j$-th excursion in it. Moreover, on $\widetilde A(n)$, the law of any subtrap $s$ is that of a Galton-Watson tree conditioned to have height strictly less than $i+1$ for some $i\leq h_n$. Then Lemma~\ref{meanreturntime} implies that for such a subtrap, $E^{\omega}[R_{s}^{j}]$ has the same law as $2\Pi_{i+1}^{i,1,1}$ which satisfies using Lemma~\ref{Z} that $E_{{\bf Q}_n}[\Pi_j^i]\leq  C (\beta {\bf f}'(q))^i$. Moreover, on $\widetilde A(n)$, there are at most $h_n^+$ vertices in the spine and at most 
$n^\epsilon$ subtraps, hence
\begin{align*}
E_{{\bf Q}_n} [\1{\widetilde A(n)}E_{\delta}^{\omega}[T_{root}^+\mid T_{root}^+<T_{\delta}^+]] &\leq h_n^+ E[G(p_{\infty})]\\& \qquad+h_n^+ n^{\epsilon} E[G(p_{\infty}/2)]  C (\beta {\bf f}'(q))^{h_n},
\end{align*}
and using $(\beta {\bf f}'(q))^{h_n} \leq C n^{(1-\epsilon)(1/\gamma-1)}$ we get
\[
E_{{\bf Q}_n}[\1{\widetilde A(n)}E_{\delta}^{\omega}[T_{root}^+\mid T_{root}^+<T_{\delta}^+]]\leq  C (\ln n) n^{(1-\epsilon)(1/\gamma -1)+\epsilon}.
\]
\end{proof}

The previous proof is mainly based on the three statements preceding the statement of Lemma~\ref{neglectup}. Similarly, one can show the following
\begin{lemma}{\label{neglectdown}}
For $z \in \{\delta,root\}$, we have
\[
E_{{\bf Q}_n}[\1{\widetilde A(n)}E_{z}^{\omega}[T_{\delta}^+\mid T_{\delta}^+<T_{root}^+]]\leq  C (\ln n) n^{(1-\epsilon)(1/\gamma -1)+\epsilon}.
\]
\end{lemma}
\begin{proof}
To apply the same methods as in the proof of Lemma \ref{neglectup}, we only need that the $h$-process corresponding to 
the conditioning 
on the event $\{T_{\delta}^+<T_{root}^+\}$ satisfies that
\begin{enumerate}
\item the walk remains unchanged on the subtraps,
\item for $y$ on the spine and $z$ a descendant of $y$ not on the spine, we have $P^{\omega}_y[X_1=\overrightarrow{y} |T_{\delta}^+<T_{root}^+]>
P^{\omega}_y[X_1 =z |T_{\delta}^+<T_{root}^+]$,
\item for $y\neq \{\delta,root\}$ on the spine, we have $P^{\omega}_y[X_1=\overrightarrow{y} |T_{\delta}^+<T_{root}^+]>\beta P^{\omega}_y[X_1=\overleftarrow{y} |T_{\delta}^+<T_{root}^+].$
\end{enumerate}

This immediately follows from the computation of the function $\widehat{h}^\omega$, given by $\widehat{h}^\omega(z) = P^{\omega}_z[T_{\delta}^+<T_{root}^+]$, 
with $\widehat{h}^{\omega}(root)=0$ and $\widehat{h}^{\omega}(\delta)=1$. A computation gives
\begin{equation}{\label{tensiontilde}}
\widehat{h}^{\omega}(y)=\widehat{h}^{\omega}(d(y\wedge \delta,\delta))=\frac{\beta^{H+1}-\beta^{d(y\wedge \delta,\delta)}}{\beta^{H+1}-1}.
\end{equation}
\end{proof}

From~(\ref{needneglectall}),~Lemma~\ref{neglectdown} and~Lemma~\ref{neglectup}, we deduce that
\begin{equation}
\label{neglectall}
\ES[\1{\widetilde A(n)} (T_{root}^+-T_{root}^{*,+})] \leq C (\ln n) n^{(1-\epsilon)(1/\gamma -1)+\epsilon}
\end{equation}
Now using~(\ref{neglectall}) and~(\ref{goal2_th_neglectall}) we prove~(\ref{goal_th_neglectall}), more precisely
\[
\text{for all $t>0$,} \qquad \PR\left[ \abs{\frac{\sum_{i=1}^{l_n} \chi_j(n)-\chi_j^*(n)}{n^{1/\gamma}}} \geq t \right]\leq o(1)+C(\ln n) n^{2\epsilon-1-\epsilon(1/\gamma-1)},
\]
and thus Proposition~\ref{th_neglectall} follows for $\epsilon<1/4$. \qed

\section{Analysis of the time spent in big traps}{\label{sectiontail}}

Let us denote $\overline{{{\bf Q}_n}}:= {\bf Q}[~\cdot \mid H=h_n^0]$ where 
\[
\label{hnulldef}
h_n^0=\lceil \ln n/-\ln {\bf f}'(q) \rceil\, .
\] 
Note that 
\begin{equation}
\label{p1H}
p_1(H) : = P^{\omega}_\delta[T_{\delta}^+<T_{root}^+]= \frac{1-\beta^{-1}}{1-\beta^{-(H+1)}}
\end{equation}
where we recall that the distance between $root$ and $\delta$ is $1+H$. Moreover let us denote 
\begin{equation}
\label{p2H}
p_2(H):=P^\omega_{\delta}[T_{root}^+ < T_{\delta}^+]=\frac{1-\beta^{-1}}{\beta^H-\beta^{-1}}.
\end{equation}

We have the following decomposition
\begin{equation}
\label{decompo_chi_star_init}
\chi_1^*(n)=\sum_{i=1}^{\text{Bin}(W_n,p_1(H))}\sum_{j=1}^{G(p_2(H))^{(i)}-1} T^{(i,j)}_{exc},
\end{equation}
where $T^{(i,j)}_{exc}$ is the time spent during the $j$-th excursion in the $i$-th trap, which is distributed under ${\bf Q}_n$ as $T_{\delta}^+$ under $P_{\delta}^{\omega}[~\cdot\mid T_{\delta}^+<T_{root}^+]$ with $\omega$ chosen according to ${\bf Q}_n$, for all $(i, j)$. The $T^{(i,j)}_{exc}$ are independent with respect to $P^{\omega}$ and for $i_1\neq i_2$  $(T^{(i_1,j)}_{exc})_{j \geq 1}$ and $(T^{(i_2,j)}_{exc})_{j \geq 1}$ are independent with respect to ${\bf Q}_n$. For $k\in \Z$ and $n$ large enough, let
$\mathcal{Z}_n^k$ be a random variable with the law of $\chi_1^{*}(n)/\beta^{H}$ under ${\bf Q}_{n+k}$ and
$\mathcal{\overline Z}_n^k$ be a random variable with the law of $\chi_1^{*}(n)/\beta^{H}$ under ${\overline {\bf Q}}_{n+k}$.
Furthermore we define
$\mathcal{Z}_{\infty}:=\frac{S_{\infty}}{1-\beta^{-1}} \sum_{i=1}^{\text{Bin}(W_{\infty},p_{\infty})} {\bf e}_i$, see (\ref{Zdef}),
where $({\bf e}_i)_{i\geq 1}$ is a family of i.i.d.~exponential random variables of parameter 1, chosen independently of the (independent) random variables $S_{\infty}$ and $W_{\infty}$. Our aim is to show the following
\begin{proposition}
\label{cvg_distribution}
We have
\[
\mathcal{\overline Z}_{n}^k \xrightarrow{d} \mathcal{Z}_{\infty}.
\]

Moreover there exists a random variable $\mathcal{Z}_{\text{sup}}$ such that 
\[
E[\mathcal{Z}_{\text{sup}}^{1-\epsilon}]<\infty, \quad \text{for any $\epsilon>0$},
\]
and 
\[
\text{for $n \in \N$ and $k> -n$,} \qquad \mathcal{\overline Z}_{n}^k \preceq \mathcal{Z}_{\text{sup}}.
\]
\end{proposition}

Let us start by proving the convergence in law. The decomposition (\ref{decompo_chi_star_init}) for $\chi_1^*(n)$ can be rewritten using~(\ref{p2H})
\begin{equation}
\label{decompo_chi_star}
\chi_1^*(n)=\beta^H\sum_{i=1}^{\text{Bin}(W_n,p_1(H))}\frac{1-\beta^{-H-1}}{1-\beta^{-1}}\frac{\beta^H-1}{\beta^H-\beta^{-1}} \frac{p_2(H)}{1-p_2(H)} \sum_{j=1}^{G(p_2(H))^{(i)}-1} T^{(i,j)}_{exc},
\end{equation}
which yields an explicit expression of $\mathcal{Z}_n^k$. We point out that $E[G(p_2(H))-1]=(1-p_2(H))/p_2(H)$. The convergence in law is due to the following facts (more precise statements follow below):
\begin{enumerate}
\item For $H$ large, 
\[
\frac{(1-\beta^{-H-1})(\beta^H-1)}{(1-\beta^{-1})(\beta^H-\beta^{-1})}\approx \frac 1 {1-\beta^{-1}},
\]
\item By the law of large numbers, we can expect 
\[
\sum_{j=1}^{G(p_2(H))^{(i)}-1} T^{(i,j)}_{exc} \approx (G(p_2(H))^{(i)}-1) E^{\omega}_\delta[T_{exc}],
\]
\item Since $p_2(H)$ is small, $(G(p_2(H))^{(i)}-1)/E[G(p_2(H))^{(i)}-1]\approx {\bf e_i}$,
\item $E^{\omega}_\delta[T_{exc}^{(1,1)}] \approx S_{\infty}$ for $H$ large enough,
\item $\text{Bin}(W_n,p_1(H)) \approx \text{Bin}(W_{\infty},p_{\infty})$ since $W_n \xrightarrow{d} W_{\infty}$ by Proposition~\ref{Winfty} and $p_1(H)\to p_{\infty}$ as $H$ goes to infinity.
\end{enumerate}

Fact $(1)$ is easily obtained, since for $\xi>0$
\begin{equation}
\label{returnapprox2}
{\bf Q}_n\Bigl[(1-\xi) \frac 1{1-\beta^{-1}}\leq \frac{1-\beta^{-H+1}}{1-\beta^{-1}} \frac{\beta^H-1}{\beta^H-\beta^{-1}}\leq \frac 1 {1- \beta^{-1}}\Bigr]=1, 
\end{equation}
for $n$ large enough. \\
We start by computations with the measure ${\bf Q}_n$ and we will be able to come back to $\overline{{\bf Q}}_{n+k}$.

For $(2)$ and $(4)$, we need to understand $P^{\omega}_{\delta}[\cdot|T_{\delta}^+<T_{root}^+]$ and to this end we will consider the $h$-process associated with this conditioning. Recall the function $\widehat{h}^{\omega}$ given as $\widehat{h}^{\omega}(z) = 
P^{\omega}_z[T_{\delta} <T_{root}]$, with $\widehat{h}^{\omega}(\delta)=1$ and $\widehat{h}^{\omega}(root)=0$, see~(\ref{tensiontilde}).

We shall enumerate the vertices of the backbone from $0$ to $H+1$, starting from $\delta$ up to $root$. With these new notations formula~(\ref{tensiontilde}) becomes
\begin{equation}
\label{tension}\widehat{h}^{\omega}(y)=\widehat{h}^{\omega}(y\wedge \delta)=\frac{\beta^{H+1}-\beta^{y\wedge \delta}}{\beta^{H+1}-1},
\end{equation}
where $y\wedge \delta$ is identified to its number which is $d(y\wedge \delta, \delta)$ as it is a vertex of the backbone.

The transition probabilities are then given as in~(\ref{hproc}). Obviously they arise from conductances, we may take
\begin{itemize}
\item[(i)] $\widehat{c}(0,1)=1$,
\item[(ii)] $\widehat{c}(i,i+1)= \widehat{c}(i-1,i) \frac{P^{\omega}_i[X_1=i+1|T_{\delta}^+<T_{root}^+]}{P^{\omega}_i[X_1=i-1|T_{\delta}^+<T_{root}^+]}$, for $1 \leq i \leq H,$
\item[(iii)] $\widehat{c}(i,z)= \widehat{c}(i,i-1)\frac{P^{\omega}_i[X_1=z|T_{\delta}^+<T_{root}^+]}{P^{\omega}_i[X_1=i-1|T_{\delta}^+<T_{root}^+]}$, for $i\neq 0$ on the spine and $z$ one of its descendants which is not on the spine,
\item[(iv)] $\widehat{c}(y,z)= \beta \widehat{c}(\overleftarrow{y},y)$ for any vertex $y$ not on the spine and $z$ one of its descendants.
\end{itemize}

We can easily deduce from this that for $y\neq root$ in the trap and denoting $z_0=\delta, \ldots, z_{n}=y$ the geodesic path from $\delta$ to $y$:
\[
\widehat{c}(z_{n-1},y)=\prod_{j=1}^{n-1}  \frac{P^{\omega}_{z_j}[X_1=z_{j+1}|T_{\delta}^+<T_{root}^+]}{P^{\omega}_{z_j}[X_1=z_{j-1}|T_{\delta}^+<T_{root}^+]},
\]
which gives using~(\ref{tension}) that
\begin{equation}
\label{chat1}\widehat{c}(i,i+1)= \beta^{-i} \frac{\widehat{h}(i+1)\widehat{h}(i)}{\widehat{h}(1)\widehat{h}(0)}=\beta^{-i}\frac{(1-\beta^{i-H})(1-\beta^{i-(H+1)})}{(1-\beta^{-H})(1-\beta^{-(H+1)})}.
\end{equation}

For a vertex $z$ not on the spine, we have 
\begin{align}
\label{chat2}\widehat{c}(z, \overleftarrow{z} )&\displaystyle{= \beta^{d(\overleftarrow{z},z\wedge \delta)} \frac{\widehat{h}(z \wedge \delta)}{\widehat{h}(z \wedge \delta-1)} c(z\wedge \delta, z\wedge \delta -1)}\nonumber \\ &\displaystyle{=\beta^{d(\overleftarrow{z},z\wedge \delta)} \frac{1-\beta^{z \wedge \delta -(H+1)}}{1-\beta^{z\wedge \delta-1 -(H+1)}} 
c(z\wedge \delta,z\wedge \delta -1 )}.
\end{align}

Together with Lemma~\ref{meanreturntime}, this yields, with $T_{exc}$ a generic random variable with the law of  $T_{exc}^{(1,1)}$,
\begin{equation}
\label{computT1}
E^{\omega}_{\delta}[T_{exc}]=2\sum_{i=0}^{H-1} \beta^{-i}\frac{(1-\beta^{i-H})(1-\beta^{i-(H+1)})}{(1-\beta^{-H})(1-\beta^{-(H+1)})} \Bigl(1+\frac{1-\beta^{i-(H+1)}}{1-\beta^{(i-1)-(H+1)}} \Lambda_i(\omega)\Bigr)
\end{equation}
where $\Lambda_i$ was defined in~(\ref{notationLi}).

We see that the random variable $S_{\infty}$ is the limit of the last quantity as $H$ goes to infinity.  More precisely, using~(\ref{computT1}) we have $0 \leq  S_{\infty}-E^{\omega}_{\delta}[T_{exc}]$ and for $n$ large enough such that
\[
\text{for all $k \leq h_n/2$,} \qquad \frac{(1-\beta^{k-h_n})(1-\beta^{k-(h_n+1)})}{(1-\beta^{-h_n})(1-\beta^{-(h_n+1)})} \geq 1- 2 \beta^{h_n},
\]
we get
\begin{align*}
S_{\infty}-E^{\omega}_{\delta}[T_{exc}] & \leq 2\Bigl(\sum_{i=0}^{h_n/2} \beta^{-i}\Bigl(1-\frac{(1-\beta^{i-H})(1-\beta^{i-(H+1)})}{(1-\beta^{-H})(1-\beta^{-(H+1)})}\Bigr)(1+\Lambda_i)\Bigr) \\ 
& \qquad \qquad \qquad \qquad + 2\Bigl(\sum_{i=h_n/2+1}^{\infty} \beta^{-i}(1+\Lambda_i)\Bigr),\\
& \leq 4 \beta^{-h_n/2} \Bigl(\sum_{i=0}^{h_n/2} \beta^{-i}(1+\Lambda_i)\Bigr) + 2\Bigl(\sum_{i=h_n/2+1}^{\infty} \beta^{-i}(1+\Lambda_i)\Bigr).
\end{align*} 

Hence, since $S_{\infty}\geq 1$ and using Chebyshev's inequality, we get
\begin{align}
\label{approx_4}
{\bf Q}_n\Bigl[(1-\xi) S_{\infty} < E^{\omega}_\delta[T_{exc}] < S_{\infty}\Bigr] & \geq 1- {\bf Q}_n[S_{\infty}-E^{\omega}_\delta[T_{exc}] \geq \xi] \\ \nonumber
& \geq 1-\frac 1 {\xi} \beta^{-h_n/2} \frac{10}{1-\beta^{-1}} \sup_{i \geq 0}E_ {\bf Q}[1+\Lambda_i]\\ \nonumber
& =1+o(1),
\end{align}
where we used Lemma~\ref{L} and the fact that $\epsilon<1/4$, this proves $(4)$.

In order to prove $(2)$, we have to bound $E_{\delta}^{\omega}[T_1^2]$ from above. This is not possible for all $\omega$, but we 
consider the event
\[
A_4(n)=\Bigl\{E^{\omega}_{\delta}[T_1^2]\leq n^{\frac{1-2\epsilon}{\gamma}}\Bigr\},
\]
and show that it satisfies the following.

\begin{lemma}{\label{A4}}
For $0<\epsilon < \min(1/3, 2\gamma/3)$, we have 
\[
{\bf Q}_n[A_4(n)^c]\to 0\, .
\]
\end{lemma}

\begin{proof}
In this proof we denote for $y$ in the trap, $N(y)$ the number of visits to $y$ during an excursion from $\delta$, which is distributed as $\card \{0\leq n \leq T_{\delta}^+: X_n=y \}$ under $P_{\delta}^{\omega}[\cdot|T_{\delta}^+<T_{root}^+]$. We have, using the Minkowski inequality,
\begin{align*}
E_{\delta}^{\omega}[T_{exc}^2] & = E_{\delta}^{\omega}\Bigl[\Bigl(\displaystyle{\sum_{y\in \text{trap}}} N(y)\Bigr)^2\Bigr] \\
                 & \leq \displaystyle{\sum_{y,z \in \text{trap}}} E_{\delta}^{\omega}[N(y)^2]^{1/2}E_{\delta}^{\omega}[N(z)^2]^{1/2}\\
                 & = \Bigl(\displaystyle{\sum_{y\in \text{trap}} }E_{\delta}^{\omega}[N(y)^2]^{1/2}\Bigr)^2.
\end{align*}

Now fix $y$ in the trap, denote $q_1=P_{\delta}^{\omega}[T_y^+<T_{\delta}^+|T_{\delta}^+<T_{root}^+]$ and $q_2=P_y^{\omega}[T_{\delta}^+<T_y^+|T_{\delta}^+<T_{root}^+]$. Then we have 
\[
\forall k\geq 1,\ P_{\delta}^{\omega}[N(y)=k]=q_1(1-q_2)^{k-1}q_2.
\]

Hence
\[
E^{\omega}_{\delta}[N(y)^2]=\sum_{n\geq 1} n^2 q_1(1-q_2)^{n-1}q_2=q_1 \frac {2-q_2}{q_2^2}\leq  \frac{2q_1}{q_2^2}.
\]

Then by reversibility of the walk, if $\widehat{\pi}$ is the invariant measure associated with the conductances $\widehat{c}$, we get $q_1=\widehat{\pi}(\delta)q_1=\widehat{\pi}(y)q_2$. This yields
\begin{equation}
\label{technical_1}
E^{\omega}_{\delta}[N(y)^2]\leq \frac{2 \widehat{\pi}(y)}{q_2}.
\end{equation}

Furthermore we have 
\begin{equation}
\label{lb_q2}
q_2\geq  (1/(Z_1(y)\beta +1))p_{\infty}\beta^{-d(\delta,\delta\wedge y)}/2.
\end{equation}

Indeed suppose that $y$ is not on the spine, otherwise the bound is simple. Starting from $y$, we reach the ancestor of $y$ with probability at least $(1/(Z_1(y)\beta +1))$ then the walker has probability at least $\beta^{-d(y,y\wedge \delta)}$ to reach $y\wedge \delta$ before $y$, next he has probability at least $1/2$ to go to $\overrightarrow{y\wedge\delta}$ before going to $z$, where $z$ is the first vertex on the geodesic path from $y\wedge \delta$ to $y$. Finally from $\overrightarrow{y\wedge\delta}$, the walker has probability at least $p_{\infty}$ to go to $\delta$ before coming back to $\overrightarrow{y\wedge\delta}$.

We denote by $\pi$ the invariant measure associated with the $\beta$-biased random walk (i.e.~not conditioned on $T_{\delta}^+<T_{root}^+$), normalized so as to have $\pi(\delta)=1$. Then we have 
\begin{enumerate}
\item For any $y$ in the trap, $\widehat{\pi}(y)\leq \pi(y)$ because of~(\ref{chat1}) and~(\ref{chat2}),
\item and by definition of the invariant measure $(Z_1(y)\beta +1) \beta^{d(\delta,\delta\wedge y)-d(y,\delta\wedge y)} = \pi(y).$
\end{enumerate}

Now plugging (2) in~(\ref{lb_q2}) yields a lower bound on $q_2$ which can be used together with (1) in~(\ref{technical_1}) to get
\[
E^{\omega}_\delta[N(y)^2] \leq C \beta^{d(\delta,\delta\wedge y)} \pi(y)^2,
\]
and
\[
E^{\omega}_\delta[T_{exc}^2]^{1/2}\leq C \sum_{y\in \text{trap}} \beta^{d(\delta,\delta\wedge y)/2} \pi(y).
\]

As a consequence, with $\widetilde A(n)$ as in Lemma~\ref{negA} we get
\begin{align*}
E_{{\bf Q}_n}[\1{\widetilde A(n)}E^{\omega}_\delta[T_{exc}^2]^{\frac 12}]&\leq C E_{{\bf Q}_n}\Bigl[\1{\widetilde A(n)} \displaystyle{\sum_{i=0}^{h_n^+} }\beta^{-i/2} \Lambda_i\Bigr]\\
                                          &\leq C \displaystyle{\sum_{i=0}^{h_n^+}} (\beta^{1/2} {\bf f}'(q))^i \\
                                          &\leq C \max(1, (\beta^{1/2} {\bf f}'(q))^{h_n^+}),
\end{align*}
where we used Lemma~\ref{meanreturntime} and Lemma~\ref{Z} for the first inequality.

Since $(\beta^{1/2} {\bf f}'(q))^{h_n^+}=n^{(1+\epsilon)(1/2\gamma -1)}$, we get by Chebyshev's inequality that
\begin{align*}
{\bf Q}_n[\1{\widetilde A(n)}E^{\omega}_\delta[T_{exc}^2]^{1/2} \geq n^{\frac{1-2\epsilon}{2\gamma}}]&\leq \frac 1 {n^{\frac{1-2\epsilon}{2\gamma}}}
E_{{\bf Q}_n}[\1{\widetilde A(n)}E^{\omega}_\delta[T_1^2]^{1/2}] \\
                                        & \leq  C\max(n^{-\frac{1-2\epsilon}{2\gamma}}, n^{3\epsilon/(2\gamma)-1-\epsilon}).
\end{align*}

The conditions on $\epsilon$ ensure that this last term goes to $0$ for $n \to \infty$. Hence 
\[
{\bf P}[\widetilde A(n)\cap A_4(n)^c] \to 0
\]
and the result follows using Lemma~\ref{negA}.
\end{proof}

We now turn to the study of 
\[
\frac{p_2(H)}{1-p_2(H)}\sum_{i=1}^{G(p_2(H))-1} T_{exc}^{(i)}.
\]

Consider the random variable
\begin{equation}
\label{var_geom}
N_g= \Bigl\lfloor \frac{-1}{\ln (1-p_2(H))} {\bf e} \Bigr\rfloor,
\end{equation}
where ${\bf e}$ is an exponential random variable of parameter 1. A simple computation shows that $N_g$ has the law of $G(p_2(H))-1$.

Set $\xi>0$, we have using Chebyshev's inequality,
\begin{align*}
       & {\bf Q}_n\Bigl[(1-\xi)N_g E^{\omega}_\delta[T_{exc}] \leq \sum_{i=1}^{N_g} T_{exc}^{(i)}\leq (1+\xi)N_g E^{\omega}_\delta[T_{exc}]\Bigr] \\
\geq & 1-{\bf Q}_n\Bigl[\Bigl\vert\frac{\sum_{i=1}^{N_g} T_{exc}^{(i)}}{N_g} -E^{\omega}_\delta[T_{exc}]\Bigr\vert>\xi E^{\omega}_\delta[T_{exc}],N_g\neq 0, E^{\omega}_\delta[T_{exc}^2]\leq n^{(1-2\epsilon)/\gamma}\Bigr] \\
    & \qquad \qquad \qquad \qquad - {\bf Q}_n[E^{\omega}[T_{exc}^2] \geq n^{(1-2\epsilon)/\gamma}]-{\bf Q}_n[N_g=0] \\
\geq & E_{{\bf Q}_n}\Bigl[\frac{n^{(1-2\epsilon)/\gamma}}{N_g} \1{N_g \neq 0} \frac 1{\xi^2}\Bigr]- {\bf Q}_n[E^{\omega}[T_{exc}^2] \geq n^{(1-2\epsilon)/\gamma}]-
{\bf Q}_n[N_g=0].
\end{align*}

We have ${\bf Q}_n[N_g=0]= p_2(H)\leq p_2(h_n) \leq C n^{-(1-\epsilon)/\gamma}\ln n$, and hence
\[
E_{{\bf Q}_n}\Bigl[\frac{\1{N_g \neq 0}}{N_g}\Bigr]=\ES\Bigl[-\frac{p_2(H)}{1-p_2(H)} \ln p_2(H)\Bigr] \leq  C n^{-(1-\epsilon)/\gamma}(\ln n)^2.
\]

Putting together the two previous equations, using Lemma~\ref{A4}, we get for $\xi <1$,
\begin{equation}
\label{LLNonreturns}
{\bf Q}_n\Bigl[(1-\xi)N_gE^{\omega}_\delta[T_{exc}] <\sum_{i=1}^{N_g} T_{exc}^{(i)}<(1+\xi)N_g E^{\omega}_\delta[T_{exc}]\Bigr] \to 1\, .
\end{equation}
This shows $(2)$.
Turning to prove $(3)$, we have
\begin{align*}
&{\bf Q}_n\Bigl[(1-\xi)\Bigl\lfloor \frac 1 {- \ln (1-p_2(H))}  {\bf e}\Bigr\rfloor \leq \frac {1-p_2(H)} {p_2(H)} {\bf e} \leq (1+\xi) \Bigl\lfloor \frac 1 {-\ln (1-p_2(H))} {\bf e} \Bigr\rfloor\Bigr] \\
\geq &1-{\bf Q}_n\Bigl[\Bigl(\frac{1-p_2(H)}{p_2(H)}-\frac {1-\xi} {-\ln (1-p_2(H))}\Bigr) {\bf e} <1\Bigr] \\ & \qquad \qquad -  {\bf Q}_n\Bigl[\Bigl(\frac{1-p_2(H)}{p_2(H)} -\frac {1+\xi} {-\ln (1-p_2(H))}\Bigr) {\bf e} >-2 \Bigr],
\end{align*}
furthermore since $\abs{\frac{1-p}p-\frac 1 {-\ln (1-p)}}$ is bounded on $(0,\epsilon_1)$ by a certain $M>0$ so that for $n$ large enough with $p_2(h_n)<\epsilon_1$, we get 
\begin{align}
\label{geom2exp}
& {\bf Q}_n\Bigl[(1-\xi)\Bigl\lfloor \frac 1 {- \ln (1-p_2(H))} {\bf e} \Bigr\rfloor \leq \frac {1-p_2(H)} {p_2(H)} {\bf e}  \leq (1+\xi) \Bigl\lfloor \frac 1 {-\ln (1-p_2(H))} {\bf e} \Bigr\rfloor\Bigr]  \\ \nonumber
\geq & 1-{\bf Q}_n\Bigl[\Bigl(\frac {\xi} {-\ln (1-p_2(H))}-M \Bigr){\bf e}<1\Bigr]-{\bf Q}_n\Bigl[\Bigl(-\frac {\xi} {-\ln (1-p_2(H))}+M \Bigr){\bf e}>-2\Bigr]  \\ \nonumber
\geq & \exp\Bigl(-\frac 2{\xi/(-\ln (1-p_2(h_n)))-M}\Bigr)\geq 1-(C/\xi)p_2(h_n),
\end{align}
which shows $(3)$.

As a consequence of ~(\ref{approx_4}),~(\ref{LLNonreturns}) and~(\ref{geom2exp}), we see that for all $\xi \in (0,1)$,
\begin{equation}
\label{returnapprox1}
{\bf Q}_n\Bigl[(1-\xi) S_{\infty} {\bf e} \leq \frac{p_2(H)}{1-p_2(H)}\sum_{i=1}^{N_g} E^{\omega}[T_{exc}^{(i)}]\leq (1+\xi) S_{\infty} {\bf e} \Bigr] \to 1
\end{equation}
for $ n \to \infty$.
Using~(\ref{p2H}),~(\ref{returnapprox1}) and~(\ref{returnapprox2}) we get
\begin{equation}
\label{returnapprox22}
{\bf Q}_n\Bigl[(1-\xi) \frac{S_{\infty} {\bf e}}{1-\beta^{-1}} \leq \frac 1{\beta^H}\sum_{i=1}^{N_g} E^{\omega}[T_{exc}^{(i)}]\leq (1+\xi) \frac{S_{\infty} {\bf e}}{1-\beta^{-1}} \Bigr] \to 1
\end{equation}
for $ n \to \infty$, which sums up $(2)$, $(3)$ and $(4)$. For any $k> -n$, the equation~(\ref{returnapprox22}) obviously holds replacing $n$ with $n+k$, 
and since \\
${\bf Q}_n[H=\lceil \ln (n+k) /\ln {\bf f}'(q) \rceil]\geq c_k>0$ (this follows from Lemma \ref{tailH}),
we have
\begin{equation}
\label{returnapprox}
\overline{{\bf Q}}_{n+k}\Bigl[(1-\xi) \frac{S_{\infty} {\bf e}}{1-\beta^{-1}} \leq \frac 1{\beta^H}\sum_{i=1}^{N_g} E^{\omega}[T_{exc}^{(i)}]\leq (1+\xi) \frac{S_{\infty} {\bf e}}{1-\beta^{-1}} \Bigr]\to 1.
\end{equation}

Only part $(5)$ remains to be shown. Coupling $\text{Bin}(W_{n},p_{\infty})$ and $\text{Bin}(W_{n},p_1(H))$ in the standard way,
\begin{align*}
&\overline{{\bf Q}}_{n+k}[\text{Bin}(W_{n},p_{\infty}) \neq \text{Bin}(W_{n},p_1(H))]\\
\leq& \sum_{j\geq 0} \PR[W_{n}=j] \overline{{\bf Q}}_{n+k}[\text{Bin}(j,p_{\infty}) \neq \text{Bin}(j,p_1(H))]\\
\leq& \sum_{j\geq 0} \PR[W_{n}=j] j \left(p_1(h^0_{n+k}) - p_{\infty})\right) \\ 
\leq& \ES[W_{n}]\left(p_1(h^0_{n+k}) - p_{\infty}\right)\\ 
\leq& C \left(p_1(h^0_{n+k}) - p_{\infty}\right) \to 0,\quad\text{ for } n \to \infty
\end{align*}
where $ C:=E[G(p_{\infty}/3)]  \geq \ES[W_{n}] $ by Lemma~\ref{tailW}. Hence,
\[
\overline{{\bf Q}}_{n+k}\Bigl[\frac 1{\beta^H}\sum_{i=1}^{\text{Bin}(W_{n},p_1(H))}\sum_{j=1}^{G(p_2(H))-1} T_{exc}^{(i,j)}\geq t\Bigr] -
\overline{{\bf Q}}_{n+k}\Bigl[\frac 1{\beta^H}\sum_{i=1}^{\text{Bin}(W_{n},p_{\infty})}\sum_{j=1}^{G(p_2(H))-1} T_{exc}^{(i,j)}\geq t\Bigr]
\to 0
\]

For $\epsilon_1>0$, introduce $N(\epsilon_1)$ such that $\max_{n\leq \infty} \PR[W_n \geq N(\epsilon_1)] \leq (1-p_{\infty}/3)^{N(\epsilon_1)} \leq  \epsilon_1$ and using the independence of $W_n$ (for $n\in \N\cup \{\infty\}$) of the trap and the walk on the trap, we get for any $\epsilon_1>0$,
\begin{align*}
     & \left \vert \overline{{\bf Q}}_{n+k}\Bigl[\frac 1{\beta^H}\sum_{i=1}^{\text{Bin}(W_{n},p_{\infty})}\sum_{j=1}^{G(p_2(H))-1} T_{exc}^{(i,j)}\geq t\Bigr] 
 - \overline{{\bf Q}}_{n+k}\Bigl[\frac 1{\beta^H}\sum_{i=1}^{\text{Bin}(W_{\infty},p_{\infty})}\sum_{j=1}^{G(p_2(H))-1} T_{exc}^{(i,j)}\geq t\Bigr]\right \vert\\
      &\leq  \left\vert \sum_{j\geq 0}(\PR[W_{n}=j]-\PR[W_{\infty}=j])  \overline{{\bf Q}}_{n+k}\Bigl[\frac 1{\beta^H}\sum_{i=1}^{\text{Bin}(j,p_{\infty})}\sum_{j=1}^{G(p_2(H))-1} T_{exc}^{(i,j)}\geq t\Bigr]\right\vert\\
      &\leq  \left\vert\sum_{j\in[0,N(\epsilon_1)]}(\PR[W_{n}=j]-\PR[W_{\infty}=j]) \overline{{\bf Q}}_{n+k}\Bigl[\frac 1{\beta^H}\sum_{i=1}^{\text{Bin}(j,p_{\infty})}\sum_{j=1}^{G(p_2(H))-1} T_{exc}^{(i,j)}\geq t\Bigr]\right\vert+ \epsilon_1\\
\end{align*}
and the right-hand side goes to $\epsilon_1$ as $n$ goes to infinity since
\[
\max_{j\leq N(\epsilon_1)}\abs{\PR[W_n=j]-\PR[W_{\infty}=j]} \to 0,
\]
by Proposition~\ref{Winfty}. So letting $\epsilon_1$ go to 0, we see that
\begin{equation}
\label{approx_234}
\overline{{\bf Q}}_{n+k}\Bigl[\frac 1{\beta^H}\sum_{i=1}^{\text{Bin}(W_{n},p_1(H))}\sum_{j=1}^{G(p_2(H))-1} T_{exc}^{(i,j)}\geq t\Bigr] 
- \overline{{\bf Q}}_{n+k}\Bigl[\frac 1{\beta^H}\sum_{i=1}^{\text{Bin}(W_{\infty},p_{\infty})}\sum_{j=1}^{G(p_2(H))-1} T_{exc}^{(i,j)}\geq t\Bigr]
\to 0
\end{equation}

Let us introduce
\begin{align*}
A(\xi)=&\Bigl\{\text{for all } i\in [1,\text{Bin}(W_{\infty},p_{\infty})],\\ 
&\ \ \frac 1{\beta^H}\sum_{j=1}^{G^{(i)}(p_2(H))-1} T_{exc}^{(i,j)}\in \Bigl[(1-\xi)\frac{S_{\infty}}{1-\beta^{-1}} {\bf e}_i,(1+\xi)\frac{S_{\infty}}{1-\beta^{-1}} {\bf e}_i\Bigr] \Bigr\},
\end{align*}
where $({\bf e}_i)_{i\geq 1}$ is a sequence of i.i.d.~exponential random variables of parameter 1 which satisfy
\[
G^{(i)}(p_2(H))-1=\Bigl\lfloor\frac{-1}{\ln(1-p_2(H))}{\bf e}_i\Bigr\rfloor.
\]

We have, denoting $o_1(1)$ the left hand side of~(\ref{approx_234})
\begin{align*}
 \overline{{\bf Q}}_{n+k}[A(\xi)] & \geq  \sum_{i \geq 0} \PR[\text{Bin}(W_{\infty},p_{\infty})=i](1-o_1(1))^i \\
&\geq  \sum_{i \geq 0} \PR[\text{Bin}(W_{\infty},p_{\infty})=i](1-i o_1(1)) \\
&=  1-\ES[W_{\infty}]o_1(1) \to 1 \quad \text{for } n \to \infty\, .
\end{align*}
Hence, for any $\xi>0$, we get
\[
\overline{{\bf Q}}_{n+k}\Bigl[\frac 1{\beta^H}\sum_{i=1}^{\text{Bin}(W_{n},p_1(H))}\sum_{j=1}^{G(p_2(H))-1} 
T_{exc}^{(i,j)}\geq t\Bigr]- 
{\PR}\Bigl[\frac{S_{\infty}}{1-\beta^{-1}}\sum_{i=1}^{\text{Bin}(W_{\infty},p_{\infty})}{\bf e}_i\geq \frac t{1+\xi}\Bigr]
\to 0 
\]
and
\[
\overline{{\bf Q}}_{n+k}\Bigl[\frac 1{\beta^H}\sum_{i=1}^{\text{Bin}(W_{n},p_1(H))}\sum_{j=1}^{G(p_2(H))-1} T_{exc}^{(i,j)}\geq t \Bigr] 
- {\PR}\Bigl[\frac{S_{\infty}}{1-\beta^{-1}}\sum_{i=1}^{\text{Bin}(W_{\infty},p_{\infty})}{\bf e}_i\geq \frac t{1-\xi}\Bigr] \to 0
\]

Concluding by using the two previous equations with $\xi$ going to $0$, we have the following convergence in law:
\[
\overline{\mathcal{Z}}_{n}^k=\frac 1{\beta^H}\sum_{i=1}^{\text{Bin}(W_n,p_1(H))}\sum_{j=1}^{G(p_2(H))-1} T_{exc}^{(i,j)} \xrightarrow{d} \frac{S_{\infty}}{1-\beta^{-1}}\sum_{i=1}^{\text{Bin}(W_{\infty},p_{\infty})}{\bf e}_i,
\]
where we recall that $\overline{\mathcal{Z}}_{n}^k$ has the law of $\chi_1^{*}(n)/\beta^{H}$ under ${\overline {\bf Q}}_{n+k}$
and the ${\bf e}_i$ are i.i.d.~exponential random variables of parameter 1. This shows the first part of Proposition~\ref{cvg_distribution}.

Now let us prove the stochastic domination part. First notice that
\[
\text{Bin}(W_n,p_1(H))\preceq G(p_{\infty}/3) \text{ and } 
E^{\omega}_\delta[T_1] \preceq T_{exc}^{\infty},
\]
where $T_{exc}^{\infty}$ is distributed as the return time to $\delta$, starting from $\delta$, on an infinite trap. Hence for $k> -n$
\[
\overline{\mathcal{Z}}_{n}^k \preceq \frac 1{\beta^{h^0_{n+k}}}\sum_{i=1}^{G(p_{\infty}/3)}\sum_{j=1}^{G(p_2(h^0_{n+k}))} T_{exc}^{\infty,(i,j)},
\]
where $(T_{exc}^{\infty,(i,j)})_{i, j \geq 1}$ are i.i.d.~copies of $T_{exc}^{\infty}$. Now recalling that $\sum_{i=1}^{G(a)} G(b)^{(i)}$ has the same law as $G(ab)$, where all geometric random variables are independent, and using the fact that
\[
\beta^{h_{n+k}^0} \geq c E[G(p_{\infty}p_2(h^0_{n+k})/3)],
\]
for some $c= c(\beta) >0$, we get
\[
\overline{\mathcal{Z}}_{n}^k \preceq \frac {C} {E[G(p_{\infty}p_2(h^0_{n+k})/3)]}\sum_{i=1}^{G(p_{\infty}p_2(h^0_{n+k})/3)}T_{exc}^{\infty, (i)}.
\]
Now, we prove the following technical lemma.
\begin{lemma}
Let $(X_i)_{i\geq 0}$ a sequence of i.i.d.~non-negative random variables such that $E[X_1]<\infty$ and set $Y_i:=(X_1+\cdots+X_i)/i$. Then there exists a random variable $Y_{\text{sup}}$ such that
\[
\text{for all $i\geq 0$,} \qquad Y_i \preceq Y_{\text{sup}},
\]
and $E[Y_{\text{sup}}^{1-\epsilon}]<\infty$ for all $\epsilon>0$.
\end{lemma}
\begin{proof}
Using Chebyshev's inequality we get that for any $i\geq 0$,
\[
\text{for all $t \geq 0$}, \qquad P[Y_i> t]\leq \frac 1t E[X_1].
\]
If we choose $Y_{\text{sup}}$ such that $P[Y_{\text{sup}}> t]=\min(1,E[X_1]/t)$ for $x\geq 0$, then $Y_{\text{sup}}$ stochastically dominates all $Y_n$ and has a finite $(1-\epsilon)$-th moment for all $\epsilon>0$. 
\end{proof}  

Now we apply this Lemma to the random variables $T_{ext}^{\infty,(i)}$ which are integrable under $\PR$ and we get a certain random variable 
$T_{\text{sup}}$. We add to our probability spaces a copy of $T_{\text{sup}}$ which is independent of all other random variables. Then for any $t\geq 0$, 
\begin{align*}
\PR [\overline{\mathcal{Z}}_{n}^k  \geq t]& \leq  \PR\Bigl[\frac C{E[G(p_{\infty}p_2(h^0_{n+k})/3)]} \sum_{i=1}^{G(p_{\infty}p_2(h^0_{n+k})/3)} T_{exc}^{\infty,(i)}\geq t \Bigr] \\
                        & \leq   \sum_{k\geq 0} \PR [G(p_{\infty}p_2(h^0_{n+k})/3)=k] \PR \Bigl[\frac C{E[G(p_{\infty}p_2(h^0_{n+k})/3)]} \sum_{i=1}^{k} T_{exc}^{\infty,(i)}\geq x \Bigr] \\
                        & \leq   \sum_{k\geq 0} \PR [G(p_{\infty}p_2(h^0_{n+k})/3)=k] \PR\Bigl[C\frac k{E[G(p_{\infty}p_2(h^0_{n+k})/3)]} T_{\text{sup}}\geq t \Bigr]\\
                        &\leq \PR \Bigl[C\frac {G(p_{\infty}p_2(h^0_{n+k})/3)}
{E[G(p_{\infty}p_2(h^0_{n+k})/3)]} T_{\text{sup}}\geq t \Bigr],
\end{align*}
and since $p_{\infty}p_2(h^0_{n+k})/3<1/3$, we can use the fact that for any $a<1/3$ we have $G(a) /E[G(a)] \preceq 3/2{\bf e}$. This shows that
\[
\text{for all $n\geq 0$, and $k > -n$, } \qquad \overline{\mathcal{Z}}_{n}^k  \preceq C {\bf e} T_{\text{sup}},
\]
where ${\bf e}$ and $T_{\text{sup}}$ are independent, so that the right-hand side has finite $(1-\epsilon)$-th moment for all $\epsilon>0$. This finishes the proof of the second part in Proposition~\ref{cvg_distribution}.\qed

\section{Sums of i.i.d.~random variables}
\label{sect_sum_iid}

This section is completely self-contained and the notations used here are not related to those used previously.

Set $\beta>1$ and let $(X_i)_{i\geq 0}$ be a sequence of i.i.d.~integer-valued non-negative random variables such that
\begin{equation}
\label{tail_X}
P[X_1 \geq n]\sim C_X \beta^{-\gamma n},
\end{equation}
for $C_X\in(0,\infty)$ and $\gamma>0$.

Let $(X_i^{(l)})_{i\geq 0}$ be a sequence of i.i.d.~integer-valued non negative random variables with the law of $X_i$ conditioned on $X_i \geq f(l)$, where $f:\N \to \N$ is such that $l-f(l) \to \infty$.

Let $(Z_i^{(l)})_{i\geq 0, l\geq 0}$ be another sequence of i.i.d.~non-negative random variables and let $Z_i^{(l),(k)}$ have the law of $Z_i^{(l)}$ under $P[~\cdot \mid X_i^{(l)}=l+k]$, if this last probability is well defined, and as $Z_i^{(l),(k)}= 0$ otherwise. Define
\begin{equation}\label{fkaelldef}
\text{for $k\in \Z,\ l \geq 0$,} \qquad \overline{F}_k^{(l)}(x):=P[Z_i^{(l),(k)}>x],
\end{equation}

We introduce the following assumptions.
\begin{enumerate}
\item There exists a certain random variable $Z_{\infty}$ such that 
\[
\text{for all $k \in \Z$ and $l\geq 0$,}\qquad  Z_i^{(l),(k)} \xrightarrow{d} Z_{\infty}.
\]
\item There exists a random variable $Z_{\text{sup}}$ such that
\[
\text{for all $l \geq 0,\ k\geq -(l-f(l))$ and $i\geq 0$,} \qquad Z_i^{(l),(k)} \preceq Z_{\text{sup}},
\]
and $E[Z_{\text{sup}}^{\gamma+\epsilon}]<\infty$ for some $\epsilon>0$.
\end{enumerate}

Moreover set
\[
Y_i^{(l)}=Z_i^{(l)} \beta^{X_i^{(l)}} \text{ and } S_n^{(l)}=\sum_{i=1}^n Y_i^{(l)},
\]
and for $\lambda > 0$, $(\lambda_l)_{l\geq 0}$ converging to $\lambda$ and $l\in \N$, define
\[
\begin{array}{lll}
N_l^{(\lambda)}&=&\left\lfloor \lambda_l^{\gamma} \beta^{\gamma (l-f(l))}\right \rfloor,\\
K_l^{(\lambda)}&=&\lambda \beta^l.
\end{array}
\]

Finally we denote by $\overline{F}_{\infty}(x)=P[Z_{\infty} > x]$ the tail function of $Z_{\infty}$.

\begin{theorem}
\label{sum_iid}
Suppose that $\gamma<1$ and Assumptions $(1)$ and $(2)$ hold true. Then we have
\[
\text{for all $\lambda > 0 $ and $(\lambda_l)_{l\geq 0}$ going to $\lambda$}, \qquad \frac{S_{N_l^{(\lambda)}}^{(l)}}{K_l^{(\lambda)}} \to \IG(d_{\lambda},0,\LR_{\lambda}),
\]
where $\IG$ is an infinitely divisible law. The L\'evy spectral function $\LR_{\lambda}$ satisfies
\begin{equation}
\label{th_spectral_1}
\text{for all $\lambda >0$ and $x\in \R$,} \qquad \LR_{\lambda}(x)=\lambda^{\gamma} \LR_1(\lambda x).
\end{equation}
and
\begin{equation}
\label{th_spectral_2}
\LR_1(x)=\left\{ \begin{array}{ll}
                         0 & \text{ if $x<0$,}\\
                         \displaystyle{-(1-\beta^{-\gamma})\sum_{k\in \Z}\beta^{\gamma k}\overline{F}_{\infty}(x \beta^k)} & \text{ if $x>0$.}
                         \end{array} \right.
\end{equation}
In particular, $\IG(d_{\lambda},0,\LR_{\lambda})$ is continuous.
Moreover, $d_\lambda$ is given by
\begin{equation}
\label{th_drift}
d_{\lambda}=\lambda^{1+\gamma}(1-\beta^{-\gamma})\sum_{k\in \Z}\beta^{(1+\gamma)k}E\Bigl[\frac {Z_{\infty}}{(\lambda\beta^k)^2+Z_{\infty}^2}\Bigr].
\end{equation}
\end{theorem}

The fact that the quantities appearing above are well defined will be established in the course of the proof.

In order to prove Theorem~\ref{sum_iid} we will apply  Theorem 4 in~\cite{BAMR}, which is itself a 
consequence of Theorem IV.6 (p.~77) in~\cite{Petrov}.
\begin{theorem}
Let $n(t):[0,\infty)\to\N$ and for each $t$ let $\{Y_k(t):1\leq k \leq n(t)\}$ be a sequence of independent identically distributed random variables. 
Assume that for every $\epsilon>0$, it is true that 
\begin{equation}
\label{th_petrov_cond_tech}
\lim_{t\to \infty} P\left[Y_1(t)>\epsilon\right]=0.
\end{equation}
Now let $\LR(x):\R \setminus \{0\} \to \R$ be a L\'evy spectral function, $d\in \R$ and $\sigma>0$. Then the following statements are equivalent:
\begin{itemize}
\item[(i)]
\[
\lim_{t\to \infty} \sum_{k=1}^{n(t)} Y_k(t)\xrightarrow{d} X_{d,\sigma, \LR} \quad\text{for } t \to\infty
\]
where $X_{d,\sigma, \LR}$ has law $\IG(d ,\sigma,\LR)$.
\item[(ii)] Define for $\tau>0$ the random variable $Z_{\tau}(t):=Y_1(t) \1{\abs{Y_1(t)}\leq \tau}$. Then if $x$ is a continuity point of $\LR$,
\[
\LR(x)= \left\{\begin{array}{ll}
                    \lim_{t\to \infty} n(t) P\left[Y_1(t)\leq x\right], & \text{for } x<0, \\
                   - \lim_{t\to \infty} n(t) P\left[Y_1(t) >  x\right], & \text{for } x>0, \\
                    \end{array}\right.
\]
\[
\sigma^2=\lim_{\tau \to 0} \limsup_{t \to \infty}\left( n(t) \text{Var}(Z_{\tau}(t))\right),
\]
and for any $\tau>0$ which is a continuity point of $\LR(x)$,
\[
d=\lim_{n\to \infty} n(t) E\left[Z_{\tau}(t)\right]+\int_{\abs{x}>\tau} \frac x {1+x^2} d\LR(x)-\int_{\tau\geq \abs{x}>0} \frac{x^3}{1+x^2} d\LR(x).
\]
\end{itemize}
\end{theorem}

The condition (\ref{th_petrov_cond_tech}) is verified in the course of the proof, in our case
$n(t)$ goes to infinity.

\subsection{Computation of the L\'evy spectral function}

Fix $\lambda \in [1, \beta)$ and assume that $x>0$ is a continuity point of $\LR_{\lambda}$. We want to show that
\begin{equation}
\label{spec_result}
-\lim_{l\to \infty} N_l^{(\lambda)}P\left[\frac{Y_1^{(l)}}{K_l^{(\lambda)}}>x\right]=\LR_{\lambda}(x).
\end{equation}

The discontinuity points of $\LR_{\lambda}$ are exactly $\mathcal{C}_{\lambda}=\{(\beta^{k} y_n)/\lambda,\ k\in \Z,\ n\in\N\}$ where $\{y_n,n\in \N\}$ are the discontinuity points of $\overline{F}_{\infty}$ (these sets are possibly empty). 

Let us introduce
\begin{equation}
\label{def_an}
\text{for $k\in \Z$,} \qquad a_k^{(l)}:=P[X_1^{(l)}\geq l+k].
\end{equation}

Since $N_l^{(\lambda)}\sim (\lambda \beta^{l-f(l)})^{\gamma} $ , we can write, recalling (\ref{fkaelldef})
\begin{align*}
               &\beta^{\gamma (l-f(l))}P\Bigl(\frac{Y_1^{(l)}}{K_l^{(\lambda)}}>x\Bigr)\\
              =&\sum_{k \in \Z} \1{k \geq -(l-f(l))} \overline{F}_{k}^{(l)}(\lambda x \beta^{-k})\beta^{\gamma (l-f(l))}\Bigl(a_{k}^{(l)}-a_{k+1}^{(l)}\Bigr).
\end{align*}

Now recalling~(\ref{tail_X}) and~(\ref{def_an}), we see that for $l \to \infty$,
\begin{equation}
\label{tail_an}
\beta^{\gamma (l-f(l))}a_{k}^{(l)} \to \beta^{-\gamma k}\, .
\end{equation}
using $l-f(l)\to \infty$, the fact that $\lambda x\beta^k$ is a continuity point of $\overline{F}_{\infty}$ (because $x>0$ is a continuity point of $\LR_{\lambda}$) for any $k$ and Assumption $(1)$, we see that for all $k\in \Z$
\begin{align}
\label{cvg_term}
& \1{k \geq -(l-f(l))} \overline{F}_{k}^{(l)}(\lambda x \beta^{-k})\beta^{\gamma (l-f(l))}\Bigr(a_{k}^{(l)}-a_{k+1}^{(l)}\Bigl) \nonumber \\
\to  & \overline{F}_{\infty}(\lambda x \beta^{-k}) \beta^{-\gamma k} (1-\beta^{-\gamma})\quad \text{for } l \to \infty\, .
\end{align}
In order to exchange limit and summation, we need to show that the terms of the sum are dominated by a function which does not depend on $l$ and is summable. Recalling Assumption $(2)$ and using~(\ref{tail_X}) we see that $\beta^{\gamma (l-f(l))}a_{k}^{(l)}\leq C_1 \beta^{-\gamma k}$ and
\begin{align*}
& \sum_{k \in \Z} \1{k \geq -(l-f(l))} \overline{F}_{k}^{(l)}(\lambda x \beta^{-k})\beta^{\gamma (l-f(l))}\Bigl(a_{k}^{(l)}-a_{k+1}^{(l)}\Bigr) \\ \leq & C \sum_{k\in \Z} \overline{F}_{\text{sup}}(\lambda x \beta^{-k}) \beta^{-\gamma k},
\end{align*}
where $\overline{F}_{\text{sup}}(x)=P[Z_{\text{sup}}>x]$. This last sum converges clearly for $k \to \infty$, and to show that it converges for 
$k \to -\infty$ we simply notice that for any $y>0$
\begin{align*}
\sum_{k > 0} E[\1{Z_{\text{sup}}> y \beta^{k}}]\beta^{\gamma k}&=E\Bigl[\sum_{0 < k \leq \left \lfloor \ln (Z_{\text{sup}}/y)/\ln \beta \right\rfloor} \beta^{\gamma k} \Bigr] \\
                &\leq (1-\beta^{-\gamma})^{-1} E[\beta^{\gamma \ln(Z_{\text{sup}}/y)/\ln \beta}] <\infty,
\end{align*}
since we assume that $E[Z_{\text{sup}}^{\gamma+\epsilon}]<\infty$.

Hence we can exchange limit and sum. Using~(\ref{cvg_term}) and the fact that $l-f(l) \to \infty$, we get
\[
-\lim_{l\to \infty} N_l^{(\lambda)}P\left[\frac{Y_1^{(l)}}{K_l^{(\lambda)}}>x\right]=-\lambda^{\gamma}(1-\beta^{-\gamma})\sum_{k\in \Z} \overline{F}_{\infty}(\lambda x \beta^{k}) \beta^{\gamma k},
\]
and, taking into account (\ref{th_spectral_1}), this proves~(\ref{spec_result}).

\subsection{Computation of $d_\lambda$}

Fix $\lambda \in [1,\beta)$. Since the integral $\int_0^{\tau} x d\LR_{\lambda}$ is well defined, it suffices to show that for all $\tau\in \mathcal{C}_{\lambda}$, $\tau>0$
\begin{equation}
\label{drift_result}
d_{\lambda}=\lim_{l\to \infty} \frac{N_l^{(\lambda)}}{K_l^{(\lambda)}}E\Bigl[Y_1^{(l)}\1{Y_1^{(l)}<\tau K_l^{(\lambda)}}\Bigr]-\int_0^{\tau}xd\LR_{\lambda}+\int_0^{\infty} \frac x {1+x^2}d\LR_{\lambda}.
\end{equation}

First let us notice that $N_l^{(\lambda)}/K_l^{(\lambda)} \sim (\lambda \beta^l)^{\gamma-1} \beta^{-\gamma f(l)}$. We introduce
\begin{equation}
\label{drift_def_W1}
\text{for all $u>0$,} \qquad G_k^{(l)}(u)=E\Bigl[Z^{(l)}_1 \1{Z_1^{(l)}\leq u}|X_1^{(l)}=k+l\bigr].
\end{equation}

Considering the first term in (\ref{drift_result}), we compute
\begin{align*}
       &\beta^{(\gamma-1)l-\gamma f(l)} E\Bigl[Y_1^{(l)}\1{Y_1^{(l)}<\tau \lambda \beta^l}\Bigr]\\
      =&\sum_{k\in \Z} \1{k\geq-(l-f(l))} \Bigl[\Bigl(a_{k}^{(l)}-a_{k+1}^{(l)}\Bigr) \beta^{\gamma(l-f(l))}\Bigr] \beta^{k} G_{k}^{(l)}(\tau \lambda \beta^{-k}).
\end{align*}

Using $l-f(l)\to \infty$,~(\ref{tail_an}) and Assumption $(1)$, we see that for all $k \in \Z$ and $\tau \in \mathcal{C}_{\lambda}$,
\begin{align}
\label{drift_cvg_term}
& \1{k\geq-(l-f(l))} \Bigl[\Bigl(a_{k}^{(l)}-a_{k+1}^{(l)}\Bigr) \beta^{\gamma(l-f(l))}\Bigr] \beta^{k} G_{k}^{(l)}(\tau \lambda \beta^{-k})
\\ \nonumber
&\to   (1-\beta^{-\gamma}) \beta^{(1-\gamma)k} G_{\infty}(\tau \lambda \beta^{-k}),
\end{align}
where
\begin{equation}
\label{def_G}
G_{\infty}(x)=E[Z_{\infty}\1{Z_{\infty} \leq x}].
\end{equation}

Once again we need to show that we can exchange limit and sum, which amounts to find a summable dominating function which does not depend on $l$. Using the fact that for $u>0$
\[
G_{k}^{(l)}(u) \leq u \text{ and } \beta^{-(\gamma-1)k}G_{k}^{(l)}(u \beta^{-k})\leq  \beta^{\epsilon k}u^{1-\gamma-\epsilon} E[Z_{\text{sup}}^{\gamma+\epsilon}],
\]
(to see the second inequality, use $E[Y\1{Y\leq s}] \leq s^a E[Y^{1-a}\1{Y\leq s}]$ with $a = 1-\gamma- \epsilon$)),
we get that
\begin{align*}
    & \sum_{k\in \Z} \1{k\geq -(l-f(l))} \Bigl[\Bigl(a_{k}^{(l)}-a_{k+1}^{(l)}\Bigr) \beta^{\gamma(l-f(l))}\Bigr] \beta^{k} G_{k}^{(l)}(\tau \lambda \beta^{-k}) \\
\leq & C \Bigl(\tau\lambda\sum_{k \geq  0} \beta^{-\gamma k} + (\tau\lambda)^{1-\gamma-\epsilon} E[Z_{\text{sup}}^{\gamma+\epsilon}] \sum_{k < 0} \beta^{\epsilon k} \Bigr)<\infty,
\end{align*}
due to Assumption $(2)$. Hence recalling~(\ref{drift_cvg_term}), we get that for $\tau \in \mathcal{C}_{\lambda}$
\begin{equation}
\label{drift_first_order}
\lim_{l\to \infty} \frac{N_l^{(\lambda)}}{K_l^{(\lambda)}}E\Bigl[Y_1^{(l)}\1{Y_1^{(l)}<\tau K_l^{(\lambda)}}\Bigr]=\lambda^{\gamma-1} (1-\beta^{-\gamma})\sum_{k\in \Z}\beta^{k(\gamma-1)} G_{\infty}(\tau \lambda \beta^k).
\end{equation}

Furthermore, recalling~(\ref{th_spectral_1}) and~(\ref{th_spectral_2}), we get for $\tau \in \mathcal{C}_{\lambda}$
\begin{align*}
\int_0^{\tau}xd\LR_{\lambda}
   &=\lambda^{\gamma}(1-\beta^{-\gamma})\int_{x\leq \tau} x\sum_{k\in \Z}\beta^{\gamma k}d(-\overline{F}_{\infty})(\lambda x\beta^k)\\
   &=\lambda^{\gamma-1}(1-\beta^{-\gamma})\sum_{k\in \Z}\beta^{(\gamma-1) k} \int_{\lambda x\beta^k \leq \lambda \tau\beta^k} \lambda x\beta^kd(-\overline{F}_{\infty})(\lambda x\beta^k)\\
   &=\lambda^{\gamma-1}(1-\beta^{-\gamma})\sum_{k\in \Z}\beta^{(\gamma-1)k}G_{\infty}(\tau\lambda\beta^k),
\end{align*}
and this term exactly compensates for~(\ref{drift_first_order}). Hence, we are left to compute in a similar fashion,
\begin{align*}
d_{\lambda}=&\int_0^{\infty}\frac x {1+x^2}d\LR_{\lambda}\\
          =&\lambda^{\gamma}(1-\beta^{-\gamma})\int_0^{\infty} \frac x{1+x^2} \sum_{k\in \Z}\beta^{\gamma k}d(-\overline{F}_{\infty})(\lambda x\beta^k) \\
          =&\lambda^{1+\gamma}(1-\beta^{-\gamma})\sum_{k\in \Z}\beta^{(1+\gamma) k} \int_0^{\infty}\frac{\lambda x\beta^k}{(\lambda\beta^{k})^2+(\lambda x\beta^k)^2}d(-\overline{F}_{\infty})(\lambda x\beta^k) \\
          =&\lambda^{1+\gamma}(1-\beta^{-\gamma})\sum_{k\in \Z}\beta^{(1+\gamma)k}E\Bigl[\frac {Z_{\infty}}{(\lambda\beta^k)^2+Z_{\infty}^2}\Bigr].\end{align*}        

This sum is finite since the terms in the sum can be bounded from above by $C_1(\lambda)\beta^{-\epsilon k} E[Z_{\text{sup}}^{\gamma+\epsilon}]$ and $C_2(\lambda) \beta^{\gamma k}$, where $C_1(\lambda) =\max_{x\geq 0} \left(x^{1-\gamma}/(\lambda^2 +x^2)\right)$ and $C_2(\lambda)=\max_{x\geq 0} x/(\lambda^2+x^2)$. The first upper bound is summable for $k \to \infty$, the other for $k\to -\infty$ and so $d_\lambda$ is well-defined.

\subsection{Computation of the variance}

We show that for any $\lambda \in [1,\beta)$ we have
\begin{equation}
\label{variance_result}
\sigma^2=\lim_{\tau \to 0}\limsup_{l\to \infty} \frac{N_l^{(\lambda)}}{(K_l^{(\lambda)})^2} \text{Var}\left(Y_1^{(l)} \1{Y_1^{(l)}\leq \tau K_l^{(\lambda)}}\right)=0.
\end{equation}

First, using~(\ref{drift_first_order}), let us notice that 
\begin{equation}
\label{variance_simplify}
\lim_{l\to \infty} \frac{N_l^{(\lambda)}}{(K_l^{(\lambda)})^2} E\left[Y_1^{(l)}\1{Y_1^{(l)}\leq \tau K_l^{(\lambda)}}\right]=0.
\end{equation}
Further, we have $N_l^{(\lambda)}/(K_l^{(\lambda)})^2 \sim (\lambda \beta^l)^{\gamma-2} \beta^{-\gamma f(l)}$. Define
\[
\text{for all $u\geq 0$,} \qquad H_k^{(l)}(u) =E\Bigl[\Bigl(Z_1^{(l)}\Bigr)^2 \1{Z_1^{(l)}\leq u} \mid  X_1^{(l)}=k+l\Bigr]. 
\]
We compute
\begin{align*}
       &\beta^{(\gamma-2)l-\gamma f(l)}E\Bigl[\Bigl(Y_1^{(l)}\Bigr)^2\1{Y_1^{(l)}<\tau \lambda \beta^l}\Bigr]\\
      =&\sum_{k\in \Z} \1{k\geq -(l-f(l))}\beta^{(\gamma-2)l-\gamma f(l)} \beta^{2(k+l)} H_{k}^{(l)}(\tau \lambda \beta^{-k})\Bigl(a_{k}^{(l)}-a_{k+1}^{(l)}\Bigr).
\end{align*}

By~(\ref{tail_X}) we have $a_{k}^{(l)} \beta^{\gamma(l-f(l))} \leq C_1 \beta^{-\gamma k}$, hence the terms of our sum are bounded above by $C_1 \beta^{(2-\gamma)k} H_{k}^{(l)}(\tau \lambda \beta^{-k})$. Note that $H_k^{(l)}(u) \leq u^2$, so that
\[
\beta^{(2-\gamma)k} H_{k}^{(l)}(\tau \lambda \beta^{-k})\leq \beta^{-\gamma k} (\tau\lambda)^2,
\]
which gives an upper bound for $k \geq 0$. On the other hand, Assumption $(2)$ implies that
\[
\beta^{(2-\gamma)k} H_{k}^{(l)}(\tau \lambda \beta^{-k}) \leq \beta^{\epsilon k} (\tau \lambda)^{2-\gamma-\epsilon} E[Z_{\text{sup}}^{\gamma +\epsilon}].
\]
These inequalities imply that
\[
\limsup_{l\to \infty} \frac{N_l^{(\lambda)}}{(K_l^{(\lambda)})^2} \text{Var}(Y_1^{(l)} \1{Y_1^{(l)}\leq \tau K_l^{(\lambda)}}) \leq C_2 \tau^{2-\gamma-\epsilon},
\]
where $C_2$ is finite and depends on $\epsilon$ and $\lambda$. Hence letting $\tau$ go to 0 yields the result, since in Assumption $(2)$ we can assume $\epsilon$ to be as small as we need in particular it can be chosen such that $2-\gamma-\epsilon>0$.

\section{Limit theorems}\label{lim-thm}

\subsection{Proof of Theorem \ref{subsequ} }\label{subseq}

Assume $\epsilon <\min (1/4,2\gamma/3)$. For $\lambda>0$, we will study the limit distributions of the hitting time properly renormalized along the subsequences defined as follows
\[
\text{for $k \in \N$,}\qquad n_{\lambda}(k)= \lfloor \lambda {\bf f}'(q)^{-k} \rfloor.
\]

First, recalling (\ref{decompo_chi_star_init}), using Proposition~\ref{cvg_distribution} and Lemma~\ref{tailH}, we can apply Theorem~\ref{sum_iid} to get
\begin{equation}
\label{cvg_chi_star}
\text{for any $(\lambda_l)_{l\geq 0}$ going to $\lambda$,}\qquad \frac 1{\lambda\beta^{k}}\sum_{i=1}^{\lfloor \lambda_{n_{\lambda}(k)}^{\gamma} \beta^{\gamma(k-f(k))} \rfloor} \chi_i^*(n_{\lambda}(k)) \xrightarrow{d} Y_{d_{\lambda},0,\LR_{\lambda}},
\end{equation}
where $f(k):=h_{n_{\lambda}(k)} = \lceil -(1-\epsilon)\ln (n_{\lambda}(k))/\ln {\bf f}'(q)\rceil $ and $ Y_{d_{\lambda},0,\LR_{\lambda}}$ is a random variable whose law $\IG(d_{\lambda},0,\LR_{\lambda})$ is the infinitely divisible law characterized by~(\ref{th_spectral_1}),~(\ref{th_spectral_2}) and~(\ref{th_drift}), where $\mathcal{Z}_{\infty}$ is given by~(\ref{Zdef}).

Using Proposition~\ref{th_neglectall},~(\ref{cvg_chi_star}) still holds if we replace $\chi_i^*(n)$ by $\chi_i(n)$. 

Recalling Proposition~\ref{approx_sum_iid} we have 
\[
\sum_{i=1}^{\lfloor(1+o_1(1)) \lambda \rho C_a {\bf f}'(q)^{-(k-h_{n_{\lambda}(k)})}\rfloor} \chi_i(n_{\lambda}(k)) \preceq \chi_{n_{\lambda}(k)} 
\preceq \sum_{i=1}^{\lfloor(1+o_2(1)) \lambda \rho C_a {\bf f}'(q)^{-(k-h_{n_{\lambda}(k)})}\rfloor}\chi_i(n_{\lambda}(k)),
\]
where 
\[
1+o_1(1)=(1-{\widetilde n}^ {-\epsilon/4})\frac{\rho_n}{\rho} \frac{n}{\lambda {\bf f}'(q)^{-k}}\frac{{\bf f}'(q)^{h_{\widetilde n}}}{{\bf f}'(q)^{h_{n}}} \text{ and } 1+o_2(1)=(1+2n^ {-\epsilon/4})\frac{\rho_n}{\rho} \frac{n}{\lambda {\bf f}'(q)^{-k}},
\]
writing $n$ for $n_{\lambda}(k)=\lfloor \lambda {\bf f}'(q)^{-k} \rfloor$ and ${\widetilde n} =n-(-2\ln n/\ln {\bf f}'(q))$.

Hence both sides of the previous equation, properly renormalized, converge in distribution to the same limit law, implying that (the law of) $\chi(n)$ converges to the same law as well. Recalling (\ref{kidef}), this yields for any $\lambda>0$ 
\[
\frac{\chi(n_{\lambda}(k))}{(\rho C_a n_{\lambda}(k))^{1/\gamma}} \xrightarrow{d} Y_{d_{(\rho C_a\lambda)^{1/\gamma}},0,\LR_{(\rho C_a\lambda)^{1/\gamma}}},
\]
where $Y_{d_{(\rho C_a\lambda)^{1/\gamma}},0,\LR_{(\rho C_a\lambda)^{1/\gamma}}}$ is a random variable with law $\IG(d_{(\rho C_a\lambda)^{1/\gamma}},0,\LR_{(\rho C_a\lambda)^{1/\gamma}})$
and we used that $\beta^{\gamma}=1/{\bf f}'(q)$.

Then by Proposition~\ref{cuts}, we get that
\[
\frac{ \Delta_{n_{\lambda}(k)}}{(\rho C_a n_{\lambda}(k))^{1/\gamma}} \xrightarrow{d} Y_{d_{(\rho C_a\lambda)^{1/\gamma}},0,\LR_{(\rho C_a\lambda)^{1/\gamma}}},
\]
which proves Theorem~\ref{subsequ}. \qed

We note for further reference that
\begin{equation}\label{contofL}
\IG(d_{\lambda},0,\LR_{\lambda}) \text{ is continuous }
\end{equation}
This follows from Theorem III.2 (p. 43) in \cite{Petrov}, since, due to (\ref{th_spectral_2}), $\lim\limits_{x \to 0}\LR_1(x) = -\infty$.

\subsection{Proof of Theorem \ref{non-conve} }\label{non-conv}

In order to prove Theorem~\ref{non-conve}, assume that\\
$(\Delta_n/n^{1/\gamma})_{n\geq 0}$ converges in law. It follows that all subsequential limits are the same, so that
\[
\text{for all $\lambda \in [1,\beta)$ and $x\in \R^{+}$,} \qquad \LR_1(x)=\lambda^{\gamma} \LR_1(\lambda x).
\]

Plugging in the values $\lambda=\beta^{1/3}$ and $x=\beta^{-2/3}$ gives
\begin{equation}
\label{nonconv0}
\sum_{k\in \Z} {\bf f}'(q)^{-k} P[\mathcal{Z}_{\infty} > \beta^{k-2/3}]={\bf f}'(q)^{-1/3}\sum_{k\in \Z} {\bf f}'(q)^{-k} P[\mathcal{Z}_{\infty} > \beta^{k-1/3}].
\end{equation}

We will show that for $\beta \to \infty$, the right hand side and the left hand side of (\ref{nonconv0}) have different limits.
First for $k\geq 1$, using Remark~\ref{Wdom}, we see that 
\begin{equation}
\label{nonconv1}
\PR[\mathcal{Z}_{\infty}>  \beta^{k- 1/3}] \leq \beta^{1/3 - k } \ES[\mathcal{Z}_{\infty}] \leq \beta^{1/3 - k} \ES[S_{\infty}]E[G(p_{\infty}/3)]=\beta^{-(k-1)}O(\beta^{-1/3}),
\end{equation}
for $\beta \to \infty$ where $O(\beta^{-1/3})=\beta^{-2/3} \ES[S_{\infty}]E[G(p_{\infty}/3)]$ does not depend on $k$ (recall Proposition \ref{Sinfty} to see that $\ES[S_{\infty}]$ is bounded in $\beta$) . In the same way,
\begin{equation}
\label{nonconvsec}
\PR[\mathcal{Z}_{\infty}> \beta^{k- 2/3}] = \beta^{-(k-1)} O(\beta^{-1/3}),
\end{equation}
for $O(\cdot)$ independent of $k\geq 1$.

Hence,
\begin{equation}
\label{posk}
\lim\limits_{\beta \to \infty} \sum_{k=1}^\infty {\bf f}'(q)^{-k} P[\mathcal{Z}_{\infty} > \beta^{k-2/3}] = 0 
= \lim\limits_{\beta \to \infty}\sum_{k=1}^\infty {\bf f}'(q)^{-k} P[\mathcal{Z}_{\infty} > \beta^{k-1/3}].
\end{equation}
For $k\leq 0$, we have 
\[
\PR[\mathcal{Z}_{\infty} > \beta^{k-1/3}] \leq \PR[\mathcal{Z}_{\infty}>0] \leq \PR[\text{Bin}(W_{\infty},p_{\infty})>0],
\]
further, since $S_{\infty} \geq 1$, we see that 
\[
\text{on $\{\mathcal{Z_{\infty}}>0\}$,} \qquad \mathcal{Z_{\infty}}\geq S_{\infty} {\bf e}_1 \geq {\bf e_1},
\]
where ${\bf e_1}$ is independent of the event $\{\mathcal{Z}_{\infty}>0\}=\{\text{Bin}(W_{\infty},p_{\infty})>0\}$. Hence
\begin{align*}
\PR[\mathcal{Z}_{\infty}> \beta^{k-1/3}]&=1-\PR[\mathcal{Z}_{\infty}\leq \beta^{k - 1/3 }]\\
                                       &\geq 1-\PR[\text{Bin}(W_{\infty},p_{\infty})=0]- P[{\bf e}_1 \leq \beta^{-1/3}] \\
                                       &= \PR[\text{Bin}(W_{\infty},p_{\infty})>0]+o(1),
\end{align*}
for $\beta\to \infty$ and hence
\begin{equation}
\label{asyklein}
\PR[\mathcal{Z}_{\infty} > \beta^{k-1/3}]= \PR[\text{Bin}(W_{\infty},p_{\infty})>0]+o(1),
\end{equation}
where $o(\cdot)$ does not depend on $k$.

In the same way,
\begin{equation}
\label{asyklein2}
\PR[\mathcal{Z}_{\infty} > \beta^{k-2/3}] =  \PR[\text{Bin}(W_{\infty},p_{\infty})>0]+o(1)
\end{equation}
for $\beta \to \infty$.
Plugging (\ref{asyklein}) and (\ref{asyklein2}) in equation~(\ref{nonconv0}) and taking into account (\ref{posk}) we see that
\[
\lim\limits_{\beta \to \infty} \sum_{k\in\Z } {\bf f}'(q)^{-k} P[\mathcal{Z}_{\infty} > \beta^{k-2/3}] 
= 
\lim\limits_{\beta \to \infty} \frac{1}{1-{\bf f}'(q)}P[\text{Bin}(W_{\infty},p_{\infty})>0]
\]
and
\[
\lim\limits_{\beta \to \infty} \sum_{k\in \Z} {\bf f}'(q)^{-k} P[\mathcal{Z}_{\infty} > \beta^{k-1/3}]
= 
\lim\limits_{\beta \to \infty}\frac{1}{1-{\bf f}'(q)}P[\text{Bin}(W_{\infty},p_{\infty})>0]\, .
\]
Hence, we would have
\[
\lim\limits_{\beta \to \infty} \frac{1}{1-{\bf f}'(q)}P[\text{Bin}(W_{\infty},p_{\infty})>0]=
\lim\limits_{\beta \to \infty} {\bf f}'(q)^{-1/3}\frac{1}{1-{\bf f}'(q)}P[\text{Bin}(W_{\infty},p_{\infty})>0]\, .
\]
This could only be possible if $P[\text{Bin}(W_{\infty},p_{\infty})>0] \to 0$ for $\beta \to \infty$, but we know that
\[
P[\text{Bin}(W_{\infty},p_{\infty})>0]> p_\infty \PR[W_{\infty}\geq 1]>c>0,
\]
where $c$ does not depend on $\beta$, see Lemma~\ref{minorWinfty}. This proves Theorem~\ref{non-conve}. \qed

In particular, if $\beta$ is large enough, 
$\IG(d_1,0,\LR_1)$ is not a stable law and this implies (vii) in Theorem \ref{infdivprop}.

\subsection{Proof of Theorem \ref{tightness}}\label{magnit}
We will show that
\begin{equation}\label{Deltatight}
\lim_{M\to \infty} \limsup_{n\to \infty} \PR\Bigl[\frac{\Delta_n}{n^{1/\gamma}}\notin [1/M,M]\Bigr]=0.
\end{equation}
This implies in particular that the family $(\Delta_n/n^{1/\gamma})_{n\geq 0}$ is tight.
We will then prove
\begin{equation}\label{Xtight}
\lim_{M\to \infty} \limsup_{n\to \infty} \PR\Bigl[\frac{|X_n|}{n^{\gamma}}\notin [1/M,M]\Bigr]=0,
\end{equation}
which implies that the family $(|X_n|/n^{\gamma})_{n\geq 0}$ is tight.
(\ref{magni}) will then shown to be a consequence of (\ref{Deltatight}) and (\ref{Xtight}).

To show (\ref{Deltatight}), note that for 
$n\in [{\bf f}'(q)^{-k},{\bf f}'(q)^{-(k+1)})$
\begin{align*}
&\PR \left[\frac{\Delta_n}{n^{1/\gamma}}\notin [1/M,M]\right]\\
\leq &\PR \left[\frac{\Delta_{{\bf f}'(q)^{-k}}}{({\bf f}'(q)^{-k}\rho C_a)^{1/\gamma}}< \frac{M}{({\bf f}'(q)\rho C_a)^{1/\gamma}}\right]+
\PR\left[\frac{\Delta_{{\bf f}'(q)^{-(k+1)}}}{({\bf f}'(q)^{k+1}\rho C_a)^{1/\gamma}}>\frac{1}{ M( {\bf f}'(q)^{-1}\rho C_a)^{1/\gamma}}]\right] \, . \\
\end{align*}
Using Theorem~\ref{subsequ} we get
\[
\limsup_n \PR\left[\frac{\Delta_n}{n^{1/\gamma}}\notin [1/M, M]\right]
\]
\[
\leq  \PR\left[Y_{d_{(\rho C_a)^{1/\gamma}},0,\LR_{(\rho C_a)^{1/\gamma}}} 
\notin \left[\frac{1}{M ({\bf f}'(q)^{-1} \rho C_a)^{1/\gamma}}, \frac{M}{({\bf f}'(q)\rho C_a)^{1/\gamma}}\right] \right]\, .
\]
where $Y_{d_{(\rho C_a)^{1/\gamma}},0,\LR_{(\rho C_a)^{1/\gamma}}}$ is a random variable with law $\IG(d_{(\rho C_a)^{1/\gamma}},0,\LR_{(\rho C_a)^{1/\gamma}})$. Here we used that the limiting law $\IG(d_x,0,\LR_x)$ is continuous, see~(\ref{contofL}), and has in particular no atom at $0$, so we get that 
\[
\lim_{M \to \infty}P[Y_{d_x,0,\LR_x} \notin[1/M,M]]=0,
\]
which proves (\ref{Deltatight}).

Let us prove (\ref{Xtight}). Let $n\geq 0$ and write $n^{\gamma}=\lambda_0 {\bf f}'(q)^{-i_0}$ for some $i_0\in \N$ and $\lambda_0 \in [1,1/{\bf f}'(q))$. Let $i\in \N$. To control the probability that $\vert X_n \vert $ is be much larger than $n^{\gamma}$, note that
\begin{align*}
\PR\Bigl[\frac{\vert X_n\vert }{n^{\gamma}} \geq \lambda_0^{-1} {\bf f}'(q)^{-i}\Bigr] & \leq \PR[\Delta_{\lfloor (\lambda_0^{-1} {\bf f}'(q)^{-i}) (\lambda_0 {\bf f}'(q)^{-i_0})\rfloor} <(\lambda_0 {\bf f}'(q)^{-i_0})^{1/\gamma}] \\
                      & = \PR\Bigl[\frac{\Delta_{\lfloor {\bf f}'(q)^{-i-i_0}\rfloor}}{(\rho C_a  {\bf f}'(q)^{-i-i_0})^{1/\gamma}} <(\lambda_0 \rho C_a {\bf f}'(q)^{-i} )^{-1/\gamma}\Bigr] .
\end{align*}

Hence for any $\epsilon>0$, and $i$ large enough such that $(\rho C_a {\bf f}'(q)^{-i} )^{-1/\gamma}<\epsilon$,
\[
\PR\Bigl[\frac{\vert X_n\vert }{n^{\gamma}} \geq {\bf f}'(q)^{-i-1}\Bigr] \leq \PR\Bigl[\frac{\vert X_n\vert }{n^{\gamma}} \geq \lambda_0^{-1} {\bf f}'(q)^{-i}\Bigr]  \leq \PR\Bigl[\frac{\Delta_{\lfloor {\bf f}'(q)^{-i-i_0}\rfloor}}{(\rho C_a  {\bf f}'(q)^{-i-i_0})^{1/\gamma}} <\epsilon\Bigr].
\]

Now, using Theorem~\ref{subsequ}, taking $n$ (i.e.~$i_0$) to infinity, we get that for any $\epsilon>0$,
\[
\text{for $i$ large enough,}\qquad \limsup_{n} \PR\Bigl[\frac{\vert X_n\vert }{n^{\gamma}} \geq {\bf f}'(q)^{-i-1}\Bigr] \leq P[Y_{d_1,0,\LR_1} \leq \epsilon],
\]
using~(\ref{contofL}) and hence
\begin{equation}
\label{limsuplimsup1}
\limsup_{M\to\infty} \limsup_{n\to \infty} \PR\Bigl[\frac{\vert X_n\vert }{n^{\gamma}} \geq M \Bigr] \leq    \limsup_{\epsilon \to 0}
P[Y_{d_1,0,\LR_1} \leq \epsilon] = 0\, .
\end{equation}

Next, we will consider the probability that $\vert X_n\vert $ is much smaller than $n^{\gamma}$. Let us denote 
\[
\text{Back}(n)=\max_{i<j\leq n} \left(\abs{X_i}-\abs{X_j}\right),
\]
the maximal backtracking of the random walk. It is easy to see that
\[
\text{Back}(n)\leq \max_{2\leq i\leq n}\left( \tau_i-\tau_{i-1}\right)\vee \tau_1 \, .
\]

Hence since $\tau_1$ and $\tau_2 - \tau_1$ have exponential moments
\begin{equation}
\label{tail_backtrack}
\PR\left[\text{Back}(n)\geq n^{\gamma/2}\right] \leq C n \exp(-c n^{\gamma/2}).
\end{equation}

If the walk is at a level inferior to $(1/M)n^{\gamma}$ at time $n$ and has not backtracked more than $n^{\gamma/2}$, it has not reached $(2/M) n^{\gamma}$. This implies that for all $M>0$,
\[
\PR\left[ \frac{\vert X_n\vert }{n^{\gamma}}<1/M\right]   \leq \PR[\text{Back}(n)\geq n^{\gamma/2}] + \PR\left[ \frac{\Delta_{\lfloor (2/M) n^{\gamma}\rfloor}}{(\rho C_a 2/M)^{1/\gamma}n}>(\rho C_a 2/M)^{1/\gamma}\right].
\]
Hence, using a reasoning similar to the proof of~(\ref{limsuplimsup1}), we have
\begin{equation}
\label{limsuplimsup2}
\lim_{M\to \infty} \limsup_{n \to \infty} \PR\left[ \frac{\vert X_n\vert}{n^{\gamma}}<1/M\right]  \leq \liminf_{M \to \infty} P[Y_{d_1,0,\LR_1} \geq M]=0.
\end{equation}

Using~ (\ref{limsuplimsup1}) and~(\ref{limsuplimsup2}), we get 
\begin{equation}
\label{limsuplimsup}
\lim_{M\to \infty} \limsup_{n \to \infty} \PR\left[ \frac{\vert X_n\vert }{n^{\gamma}}\notin [1/M,M]\right] =0,
\end{equation}
which shows (\ref{Xtight}) in Theorem~\ref{tightness}.

Let us prove (iii) in Theorem~\ref{tightness}. We have
\[
\PR\left[\lim\limits_{ n\to \infty} \frac{\ln\vert X_n\vert}{\ln n}\neq \gamma\right]\leq
\PR\left[ \limsup\limits_{ n\to \infty} \frac{\ln \vert X_n\vert }{\ln n} > \gamma
\right] + \lim\limits_{M \to \infty}
\PR\left[ \liminf\limits_{ n\to \infty} \frac{\vert X_n\vert }{n^{\gamma}} \leq
\frac{1}{M} \right]
\]
Using Fatou's Lemma,
\[
\PR\left[ \liminf\limits_{ n\to \infty} \frac{\vert X_n\vert }{n^{\gamma}} <
\frac{1}{M} \right] \leq \liminf_{n\to \infty} \PR\left[ \frac{\vert
X_n\vert}{n^{\gamma}}<1/M\right],
\]
and taking $M$ to infinity we get
\[
\PR\left[\lim\limits_{ n\to \infty} \frac{\ln\vert X_n\vert}{\ln n}\neq \gamma\right]\leq
\PR\left[ \limsup\limits_{ n\to \infty} \frac{\ln \vert X_n\vert }{\ln n} > \gamma
\right].
\]

Set $\epsilon>0$, we have
\begin{align*}
\PR\left[ \limsup\limits_{ n\to \infty} \frac{\ln \vert X_n\vert }{\ln n} >
(1+2\epsilon)\gamma \right]
&\leq \PR\left[\limsup\limits_{n\to \infty} \frac{\vert X_n
\vert}{n^{(1+\epsilon)\gamma}}\geq 1\right] \\
&\leq \PR\left[\limsup\limits_{n\to \infty} \frac{\sup_{i\leq n} \vert X_i
\vert}{n^{(1+\epsilon)\gamma}} \geq 1\right].
\end{align*}

Define
\[
D'(n)=\Bigl\{ \max_{\displaystyle{\ell\in \cup_{i=0, \ldots, \Delta_n^Y}L_{Y_i}}} H(\ell) \leq \frac {4\ln n} {- \ln {\bf f}'(q)}\Bigr\}.
\]

Denoting $t(n)$ such that $\sigma_{\tau_{t(n)}}\leq n < \sigma_{\tau_{t(n)+1}}$, we have
\begin{equation}
\label{maj1fin}
\text{for $\omega \in D'(n)$,}\qquad \qquad \vert X_{\sigma_{\tau_{t(n)}}} \vert \leq
\vert X_n\vert \leq \vert X_{\sigma_{\tau_{t(n)+1}}} \vert +\frac {4\ln n}{-\ln {\bf
f}'(q)},
\end{equation}
and since using $B_1(n)$ defined right above Lemma~\ref{B1}, we get
\begin{equation}
\label{maj2fin}
\text{for $\omega \in B_1(n)$,} \qquad \qquad \vert X_{\sigma_{\tau_{t(n)+1}}} \vert \leq
\vert X_{\sigma_{\tau_{t(n)}}} \vert +n^{\epsilon}.
\end{equation}

We have using Lemma~\ref{A1} and~(\ref{tailmaj})
\begin{align*}
\PR[D'(n)^c]&\leq \PR[A_1(n)^c] + \PR\Bigl[A_1(n),\card \cup_{i=1}^{\Delta_n^Y} L_{Y_i} > n^2 \Bigr] + \PR\Bigl[\card \cup_{i=1}^{\Delta_n^Y} L_{Y_i} \leq n^2,D'(n)^c\Bigr] \\
&\leq O(n^{-2})+ \PR\Bigl[\sum_{i=0}^{C_1n} \card L_0^{(i)} > n^2\Bigr] + n^2 {\bf Q}\Bigl[H\geq \frac{4 \ln n}{- \ln {\bf f}'(q)}\Bigr]\\
&\leq O(n^{-2})+ n^{-4} \text{Var}\Bigl(\sum_{i=0}^{C_1n} \card L_0^{(i)}\Bigr)+ n^2n^{-4}=O(n^{-2}),
\end{align*}
where we used that $\card L_0^{(i)}$ are i.i.d.~random variables which are $L^2$ since they are stochastically dominated by the number of offspring $Z$ which is $L^2$ by our assumption.

By~Lemma \ref{B1}, the previous estimate and Borel-Cantelli we have $\omega\in B_1(n)\cap
D(n)$ asymptotically, we get recalling~(\ref{maj1fin}) and~(\ref{maj2fin}) that for $\epsilon<\gamma$
\[
\PR\left[\limsup\limits_{n\to \infty} \frac{\sup_{i\leq n} \vert X_i
\vert}{n^{(1+\epsilon)\gamma}}\geq 1\right]
\leq \PR\left[\liminf \limits_{n\to \infty}  \Bigl(\frac{\vert X_{\sigma_{\tau_{t(n)}}}
\vert}{n^{(1+\epsilon)\gamma}}+o(1)\Bigr) \geq 1\right].
\]

Since $\vert X_{\sigma_{\tau_{t(n)}}}\vert \leq \vert X_n \vert$ we have
\begin{align*}
    \PR\left[\lim_{n\to \infty} \frac{\sup_{i\leq n} \vert X_i
\vert}{n^{(1+\epsilon)\gamma}}\geq 1\right]& \leq \PR\left[\liminf \limits_{n\to \infty} 
\frac{\vert X_{\sigma_{\tau_{t(n)}}} \vert}{n^{(1+\epsilon)\gamma}}\geq 1\right] \\
& \leq\liminf\limits_{n\to \infty} \PR\left[ \frac{\vert X_{n}
\vert}{n^{(1+\epsilon)\gamma}}\geq 1\right]\\
&\leq \liminf\limits_{M\to \infty}  \liminf\limits_{n\to \infty}  \PR\left[ \frac{\vert
X_{n} \vert}{n^{\gamma}}\geq M\right]=0,
\end{align*}
where we used Fatou's Lemma and (ii) in Theorem~\ref{tightness}.

Now since this result is true for all $\epsilon>0$ small enough we get
\begin{align*}
\PR\left[\lim\limits_{ n\to \infty} \frac{\ln\vert X_n\vert}{\ln n}\neq \gamma\right]
\leq&
\PR\left[ \limsup\limits_{ n\to \infty} \frac{\ln \vert X_n\vert }{\ln n} > \gamma
\right] \\
\leq& \liminf \limits_{\epsilon \to 0} \PR\left[ \limsup\limits_{ n\to \infty} \frac{\ln
\vert X_n\vert }{\ln n} > (1+2\epsilon)\gamma \right] =0,
\end{align*}
which finishes the proof of~(\ref{magni}). \qed

\subsection{Proof of Theorem \ref{infdivprop} }\label{limprop}
It remains to show (iv), (v), (vi) and (viii) in Theorem \ref{infdivprop}. 
\begin{proof} We start by proving (viii). Recall
\[
\mathcal{Z}_{\infty}=\frac{S_{\infty}}{1-\beta^{-1}} \sum_{i=1}^{\text{Bin}(W_{\infty},p_{\infty})} {\bf e}_i,
\]
and in particular the fact that $S_{\infty}$, $W_\infty$ and the i.i.d.~exponential random variables ${\bf e}_i$ are independent.
Let $\widetilde{S}_\infty = S_\infty/p_\infty$ and denote its law by $\nu_\infty$. Further, let
$\alpha_k = \PR[\text{Bin}(W_{\infty},p_{\infty})= k]$, $k=0,1,2, \ldots $. Conditioned on $\widetilde{S}_\infty $ and 
$\text{Bin}(W_{\infty},p_{\infty})$, the law of $\mathcal{Z}_{\infty}$ is a Gamma distribution. More precisely, for any test function $\varphi$,
\begin{align*}
\ES[\varphi(\mathcal{Z}_\infty)]&=
\alpha_0\varphi(0) + \sum\limits_{k=1}^\infty \int\limits_0^\infty\left(\int\limits_0^\infty \varphi(su) e^{-u}
\frac{u^{k-1}}{(k-1)!}du\right)\nu_\infty(ds)\alpha_k \\
&= \alpha_0\varphi(0) + \sum\limits_{k=1}^\infty \int\limits_0^\infty\left(\int\limits_0^\infty \varphi(v) e^{-v/s}
\frac{v^{k-1}}{(k-1)!}\frac{1}{s^k} dv\right)\nu_\infty(ds)\alpha_k \\
&= \alpha_0\varphi(0) + \int\limits_0^\infty \varphi(v)\sum\limits_{k=1}^\infty 
\alpha_k \frac{v^{k-1}}{(k-1)!}E_{\bf Q}\left[e^{-v/\widetilde{S}_\infty}\left(\widetilde{S}_\infty\right)^{-k}\right] dv
\end{align*}

We point out that, due to Lemma \ref{minorWinfty}, we have $0 < \alpha_0 < 1$.
Hence, $\mathcal{Z}_\infty$ has an atom of mass $\alpha_0$ at $0$ and the conditioned law of $\mathcal{Z}_\infty$, conditioned on $\mathcal{Z}_\infty > 0$, has the density $\psi$, where
\begin{align*}
\psi(v) &= \sum\limits_{k=1}^\infty \alpha_k \frac{v^{k-1}}{(k-1)!} E_{\bf Q}\left[e^{-v/\widetilde{S}_\infty}\left(\widetilde{S}_\infty\right)^{-k}\right]\\
&= E_{\bf Q}\left[\frac{1}{\widetilde{S}_\infty}e^{-v/\widetilde{S}_\infty}
\sum\limits_{k=1}^\infty \frac{\alpha_k}{(k-1)!}\left(\frac{v}{\widetilde{S}_\infty}\right)^{k-1}\right]
\end{align*}

Using the fact that $S_\infty \geq 2$ and $\limsup \frac{1}{k} \log \alpha_k < 0$ (see Lemma \ref{tailW}), we see that 
$\psi$ is bounded and $C_\infty$. Note that
since $S_\infty $ and $W_\infty$ have finite expectation, $\mathcal{Z}_\infty$ has also finite expectation and in particular
\begin{equation}\label{weakb}
\int\limits_0^\infty v  \psi(v)dv < \infty \, .
\end{equation}
This shows (viii) in Theorem \ref{infdivprop}. We will later need that
\begin{equation}\label{derivbound}
\int\limits_0^\infty v^{1+\gamma}  \vert \psi^\prime(v)\vert dv < \infty \, .
\end{equation}
To show (\ref{derivbound}), note that $\psi^\prime(v)$ equals
$$
 E_{\bf Q}\left[\frac{-1}{(\widetilde{S}_\infty) ^2}e^{-v/\widetilde{S}_\infty}
\sum\limits_{k=1}^\infty \frac{\alpha_k}{(k-1)!}\left(\frac{v}{\widetilde{S}_\infty}\right)^{k-1}
+ \frac{1}{(\widetilde{S}_\infty)^2}e^{-v/\widetilde{S}_\infty} 
\sum\limits_{k=2}^\infty \frac{\alpha_k}{(k-2)!}\left(\frac{v}{\widetilde{S}_\infty}\right)^{k-2}
\right]
$$
which implies, with $\alpha : = \limsup\left(\alpha_k^{1/k}\right) < 1$,
$$
\vert \psi^\prime(v)\vert \leq C_1 E_{\bf Q}\left[\frac{1}{(\widetilde{S}_\infty) ^2}e^{-(1-\alpha) v/\widetilde{S}_\infty}\right]\, .
$$
But, for $\delta \in (0,1)$,
$$
E_{\bf Q}\left[\frac{1}{(\widetilde{S}_\infty) ^2}e^{-(1-\alpha) v/\widetilde{S}_\infty}\right] 
$$
\begin{align*}
&\leq
E_{\bf Q}\left[e^{-(1-\alpha) v/\widetilde{S}_\infty}\1{\widetilde{S}_\infty < v^{1-\delta}} \right]
+ E_{\bf Q}\left[\frac{1}{(\widetilde{S}_\infty) ^2} \1{\widetilde{S}_\infty \geq v^{1-\delta}} \right]\\
&\leq
e^{-(1-\alpha) v^\delta}  + \frac{1}{v^{3-3\delta}}  E_{\bf Q}\left[\widetilde{S}_\infty\right]
\end{align*}
Now, choosing $\delta$ small enough such that $3\delta + \gamma < 1$ yields (\ref{derivbound}).

We next show that the function $\LR_1$ is absolutely continuous.  
Recalling (i) in Theorem \ref{infdivprop} we see that, for $x > 0$,
\begin{align*}
- (1-\beta^{-\gamma})^{-1} \LR_1(x) &= \sum_{k\in \Z}\beta^{\gamma k}\overline{F}_{\infty}(x \beta^k)\\
&= \sum_{k\in \Z}\beta^{\gamma k} \int\limits_{x\beta^k}^\infty \psi(v) dv \\
&= \int\limits_0^\infty \left( \sum_{k\in \Z}\beta^{\gamma k}\1{v \geq x\beta^k}\right)\psi(v) dv\, .
\end{align*}

Now,
\[
\sum_{k\in \Z}\beta^{\gamma k}\1{v \geq x\beta^k} = \sum\limits_{k \leq K(v/x)}\beta^{\gamma k} 
=:g\left(\frac{v}{x}\right)
\]
where, setting $u= \frac{v}{x}$,  $K(u)= \lfloor \frac{\log u} {\log \beta}\rfloor$. An easy computation gives
\begin{equation}\label{gform}
g(u) = \frac{\beta^{\gamma(K(u)+1)}}{\beta^\gamma -1}\, .
\end{equation}
Hence, for $x > 0$,
\begin{equation}\label{Lexpress}
- (1-\beta^{-\gamma})^{-1} \LR_1(x) = \int\limits_0^\infty g\left(\frac{v}{x}\right) \psi(v) dv = x \cdot \int\limits_0^\infty g(u) \psi(xu)du\, .
\end{equation}
The last formula shows, noting that $g(u)$ is of order $u^\gamma$ for $u \to \infty$ and recalling (\ref{weakb})
and (\ref{derivbound}), 
that $\LR_1$ is $C_1$ and in particular absolutely continuous. Due to the scaling relation (ii), the same holds true for $\LR_\lambda$. This shows (iv) in Theorem \ref{infdivprop}.

Due to (\ref{gform}), we have
\[
\frac{1}{\beta^{\gamma} -1}u^\gamma \leq g(u) \leq \frac{\beta^\gamma}{\beta^{\gamma} -1}u^\gamma
\]
Plugging this into the first equality in (\ref{Lexpress}) yields (\ref{Lbounds}).
This proves (v) in Theorem \ref{infdivprop}. To show (vi), we use a result of \cite{Tucker} which says that an infinite divisible law is absolutely continuous if the absolutely continuous component $\LR^{\text{ac}}$ of its L\'evy spectral function satisfies $\int\limits_{-\infty}^\infty d\LR^{\text{ac}}(x)= \infty$, see also \cite{Petrov}, p. 37. In our case, this is satisfied since $\LR_1^{\text{ac}}(x) = \LR_1(x)$ and $\lim\limits_{x \to 0} \LR_1(x) = -\infty$. Further, the statement about the moments of $\mu_\lambda$ follows from the corresponding statement about the moments of $\LR_\lambda$, see \cite{Ramachandran} or \cite{Petrov}, p. 36.
\end{proof}

\section*{Acknowledgements}

A.F. would like to thank his advisor Christophe Sabot for his support and many useful discussions.

We are grateful to the Institut Camille Jordan, Universit\'e de Lyon 1, the Institut f\"ur Mathematische Statistik, Universit\"at M\"unster, and Courant Institute of Mathematical Sciences, New York University, for giving A.F. the possibility to visit M\"unster and A.F. and N.G. the possibility to visit New York.

We also thank the referee for his careful reading.

\end{document}